\documentclass{amsart}

\usepackage{geometry} 
\usepackage{amssymb, mathrsfs}
\usepackage{relsize}
\usepackage[all]{xy}
\usepackage{xcolor}
\usepackage{dsfont}
\usepackage[hidelinks]{hyperref}
\usepackage{cleveref}

\numberwithin{equation}{section}

\newtheorem{theorem}[subsubsection]{Theorem}
\newtheorem{proposition}[subsubsection]{Proposition}

\newtheorem{corollary}[subsubsection]{Corollary}
\newtheorem{lemma}[subsubsection]{Lemma}
\newtheorem{conjecture}[subsubsection]{Conjecture}
\newtheorem*{theorem*}{Theorem}

\theoremstyle{definition}
\newtheorem{definition}[subsubsection]{Definition}

\theoremstyle{remark}
\newtheorem{example}[subsubsection]{Example}
\newtheorem{remark}[subsubsection]{Remark}

\title[]{Maximal ideals in the finitary symmetric group algebra in characteristic two}
\author{Kevin Coulembier}
\address{School of Mathematics and Statistics, University of Sydney, Australia}
\email{kevin.coulembier@sydney.edu.au}

\newcommand{\Sym}{\mathtt{Sym}}
\newcommand{\w}{\mathlarger{\wedge}}

\newcommand{\bk}{\Bbbk}
\newcommand{\FPdim}{\mathrm{FPdim}}

\newcommand{\unit}{\mathds{1}}

\newcommand{\sVec}{\mathtt{sVec}}
\newcommand{\Ver}{\mathtt{Ver}}
\newcommand{\Rep}{\mathtt{Rep}}

\newcommand{\Tilt}{\mathtt{Tilt}}

\newcommand{\Ann}{\operatorname{Ann}}
\newcommand{\GL}{\operatorname{GL}}
\newcommand{\Indu}{\operatorname{Ind}}
\newcommand{\Res}{\operatorname{Res}}
\newcommand{\SL}{\operatorname{SL}}
\newcommand{\ch}{\operatorname{ch}}

\newcommand{\mZ}{\mathbb{Z}}
\newcommand{\mN}{\mathbb{N}}
\newcommand{\mR}{\mathbb{R}}
\newcommand{\mC}{\mathbb{C}}

\newcommand{\mF}{\mathbb{F}}

\newcommand{\bN}{\mathbf{N}}
\newcommand{\bI}{\mathbf{I}}

\newcommand{\bS}{\mathbf{S}}
\newcommand{\bJ}{\mathbf{J}}
\newcommand{\bK}{\mathbf{K}}
\newcommand{\pA}{\mathscr{A}}
\newcommand{\pB}{\mathscr{B}}

\newcommand{\cM}{\mathcal{M}}
\newcommand{\cI}{\mathcal{I}}
\newcommand{\cJ}{\mathcal{J}}
\newcommand{\cK}{\mathcal{K}}
\newcommand{\cC}{\mathcal{C}}
\newcommand{\cD}{\mathcal{D}}

\newcommand{\OB}{\mathcal{O}\mathcal{B}}

\newcommand{\ev}{\mathrm{ev}}

\newcommand{\Hom}{\operatorname{Hom}}
\newcommand{\End}{\operatorname{End}}

\newcommand{\Ob}{\operatorname{Ob}}

\newcommand{\Irr}{\operatorname{Irr}}

\newcommand{\tto}{\twoheadrightarrow}

\renewcommand{\part}{\mathscr{P}}
\newcommand{\ppart}{\mathscr{P}^\reg}
\newcommand{\gen}{\mathscr{G}}
\newcommand{\Sch}{\operatorname{S}}
\newcommand{\reg}{\mathrm{reg}}
\newcommand{\Tr}{\mathrm{Tr}}
\newcommand{\TL}{\mathrm{TL}}

\newcommand{\fm}{\mathfrak{m}}
\begin{document}

\begin{abstract}
In 1996, Baranov and Kleshchev classified maximal ideals in the finitary symmetric group algebra over fields of characteristic not 2. The corresponding question in characteristic~2 has remained open since. In the current paper we solve the problem by establishing part of a conjectural connection between prime ideals in this group algebra and the recently defined higher Verlinde categories. To achieve this we investigate thick tensor ideals of tilting modules of the general linear group in characteristic 2.
\end{abstract}


\maketitle

\vspace{-2mm}
\begin{center}
{\textbf{To Sasha Kleshchev on his 60th birthday with admiration.}}
\end{center}
\vspace{2mm}

\section*{Introduction}

The union of the finite symmetric groups is the finitary symmetric group $S_\infty$, which appears frequently in representation theory and combinatorics. It is a natural question what maximal (or prime) ideals exist in $\bk S_\infty$, for a field $\bk$. A classical fact \cite{FL} states there are only two maximal ideals in characteristic zero: the annihilator ideals of the two one-dimensional representations of~$S_\infty$. 

If the characteristic of $\bk$ is $p>2$, Baranov and Kleshchev proved in \cite{BK} that there are precisely $p-1$ maximal ideals. We describe some alternative characterisations of these ideals below. The principal methods in \cite{BK} are a reformulation in terms of \emph{inductive systems} \cite{Za}, Kleshchev's branching rules from \cite{Kl-branching} and the Mullineux involution. The latter tool is not available for $p=2$, and consequently that classification remained elusive for 30 years. Moreover, the naive analogue of the result for $p>2$ is incorrect. Indeed, for instance by \cite{Poly} there are two known maximal ideals when $p=2$ and our main result, proved in Theorem~\ref{thm:main}, states that that list is exhaustive.

\medskip

{\bf Theorem A.} There are precisely two maximal ideals in $\bk S_\infty$ when $\mathrm{char}(\bk)=2$.

\medskip

It is enlightening to place this result in the following context. A `first order approximation' of the symmetric group algebra in characteristic $p$ is given by the Hecke algebra at a complex $p$-th root of unity. As we point out in Theorem~\ref{thm:Hecke} and Proposition~\ref{prop:Hecke}, there are precisely $p-1$ maximal ideals in the corresponding infinite Hecke algebra $H_\infty$, even when $p=2$. Hence, the classification of maximal ideals in $\bk S_\infty$ for $p>2$ is captured by the first order behaviour. However, for $p=2$, an extra maximal ideal in $\bk S_\infty$ is produced by a second order principle. That we can concretely refer to this correction as second order follows for instance from the connection with the \emph{second} Verlinde category $\Ver_4$ in the chain $\Ver_{2^n}$, $n>0$ from~\cite{BE}, see \cite{Poly} or below. For completeness, we also show in Theorem~\ref{thm:Alt} that for $p=2$ the finitary alternating group algebra $\bk A_\infty$ also has precisely two maximal ideals, where it was proved in \cite{BK} that there are $(p-1)/2$ when $p>2$.

\subsection*{Descriptions of maximal ideals} Now we describe the maximal ideals in $\bk S_\infty$ in a complementary sense to the treatment in the main part of the paper.
Let $\bk$ be a field of characteristic $p>0$.
For any $z\in \mF_p^\times$, we consider the `Markov trace' on $\bk S_\infty$ given by the linear functional
\[\Tr_z:\bk S_\infty\to \bk,\quad\mbox{satisfying}\quad \Tr_z(1)=1,\quad \Tr_z(ab)=\Tr_z(ba)\quad\mbox{and}\quad \Tr_z(s_nx)=z^{-1}\,\Tr_z(x),\]
where $s_n$ is the permutation $(n,n+1)$ and $x\in \bk S_n$. 
For $1\le m<p$ denote by $J_m$ the radical of~$\Tr_{m}$. Then $J_1,\ldots, J_{p-1}$ are the maximal ideals in $\bk S_\infty$ for $p>2$. This follows from the analogue of \cite{We}, where these Markov traces are considered for the finitary Hecke algebra, in relation with the analogue of the semisimple inductive systems from \cite{CS, Ma}, used in the original description of the maximal ideals in \cite{BK}. Alternatively, one can use \cite{Poly} and the oriented Brauer category.


Furthermore, as proved in \cite[Proposition~6.2.2]{Poly}, while $J_1$ and $J_{p-1}$ are the codimension-one ideals that we mentioned in characteristic 0, the ideals $J_m$ for $1<m<p-1$ are generated by the symmetriser of $\bk S_{p+1-m}$ and the skew symmetriser of $\bk S_{m+1}$.

In particular, for $J_2$ it is natural to consider the infinite Temperley-Lieb algebra $\TL_\infty(2)$, the quotient of $\bk S_\infty$ by the skew symmetriser of $\bk S_3$. Indeed, when $p>2$ this Temperley-Lieb algebra has a maximal ideal generated by the Jones-Wenzl idempotent in degree $p-1$, besides the obvious augmentation ideal. The inverse images of these two ideals are $J_2$ and $J_1$ respectively.

It turns out that, for $p=2$, there is still a second maximal ideal in $\TL_\infty(2)$, besides the augmentation ideal, namely the ideal generated by the Jones-Wenzl idempotent in degree $3=p^2-1$. Its inverse image in $\bk S_\infty$ is the second maximal ideal from Theorem~A, see \cite[\S 6.2]{Poly} for an alternative construction. One can show that this ideal is {\bf not} the radical of a trace functional as above. Indeed, by \cite[\S 6.2]{Poly}, the quotient by this maximal ideal is a direct limit of matrix algebras $\mathrm{Mat}_{2^j}(\bk)$ with block diagonal embeddings. Since every trace on a matrix algebra must be a scalar multiple of the standard one it follows that the trace of an element in $\mathrm{Mat}_{2^j}(\bk)$ inside $\mathrm{Mat}_{2^{j+1}}(\bk)$ is zero, so the quotient admits only the zero trace. 

\subsection*{Small inductive systems and higher Verlinde categories}

Following \cite{Za, BK}, maximal ideals in $\bk S_\infty$ are in bijection with minimal inductive systems. The latter are coherent collections of simple representations of $S_n$ for all $n$. More generally, \emph{prime} ideals in $\bk S_\infty$ relate (albeit not necessarily bijectively) to 
\emph{indecomposable} inductive systems. In \cite{Tprime} a more restrictive notion, of T-prime ideals and corresponding T-indecomposable inductive systems, was introduced.

As discussed in \cite{Tprime} or in \S \ref{sec:indsys} below, objects in $\bk$-linear symmetric monoidal categories are sources of ideals in $\bk S_\infty$ and consequently also inductive systems. In \cite[Conjecture~5.1.2]{Tprime} it was conjectured that this connection should lead to a bijection between the set of T-indecomposable inductive systems and the set of (isomorphism classes of) semisimple objects in the symmetric monoidal category $\Ver_{p^\infty}$, which plays a central role in the ongoing investigations into pretannakian categories, see \cite{BE, BEO, AbEnv, CEO} and references therein. From the validity of \cite[Conjecture~5.1.2]{Tprime} one could quickly deduce Theorem A. Our strategy for proving Theorem~A can  therefore be summarised as proving `enough' of the conjecture. Concretely, it is shown in \cite{BK} that, for every characteristic $p>0$, the only candidates for minimal inductive systems can be labelled as $\Phi[m]$ for $m\in\mZ_{>0}$. Theorem A then follows from the following, proved in Theorem~\ref{thm:conj}.

\medskip
{\bf Theorem B.} If $p=2$, then each inductive system $\Phi[m]$ is the inductive system of a semisimple object in $\Ver_{2^\infty}$ in accordance with \cite[Conjecture~5.1.2]{Tprime}.

\medskip

For $p>2$, minimal inductive systems were classified via branching rules, hence by working with individual representations of $S_n$. Our methods for Theorem~A on the other hand establish tools for proving inclusions between inductive systems, without zooming in on the individual representations.

We also provide further evidence towards \cite[Conjecture~5.1.2]{Tprime} for all $p$, by showing in Theorem~\ref{thm:conj} that all T-indecomposable inductive systems of `length' 2 and 3 correspond to $\Ver_{p^\infty}$ as predicted. The same effort also proves \cite[Conjecture~6.3.4]{Tprime} stating that all T-prime ideals in $\bk S_\infty$ admit a `categorical dimension', see Theorem~\ref{thm:primedim}.

The paper is organised as follows. In Section~\ref{sec:prel} we recall some preliminary notions. In Section~\ref{sec:indsys} we discuss inductive systems and their connection with monoidal categories. Section~\ref{sec:GL} is the most technical and arduous part of the paper, where we prove some results about thick tensor ideals in the category of tilting modules for $\GL_m$ in characteristic 2. If we were after similar results in characteristic $p>m$, we could rely on classical work of Andersen, Gelfand and Kazhdan \cite{An, GK}. To get such results for $p=2$ involves a number of new constructions and ideas which we hope will have further applications. In Section~\ref{sec:max} we discuss the classification of maximal ideals in symmetric and alternating group algebras and Hecke algebras. In Section~\ref{sec:allp} we return to arbitrary positive characteristic $p>0$ and show how to connect the study of thick tensor ideals in the monoidal subcategory of $\Tilt\GL_m$ generated by the natural representation, with the same study in the entire category. This leads to the aforementioned progress on \cite[Conjectures~5.1.2 and~6.3.4]{Tprime}.

\subsection*{Going to A without B} \cite[Conjecture~5.1.2]{Tprime} provided crucial guidance for proving Theorem~A. However, one needs strictly less than Theorem~B to prove Theorem~A, so one can discard the hardest parts of the technical Section~\ref{sec:GL}. For readers interested in the most direct proof of Theorem~A, we point out that the essential non-minimality of $\Phi[m]$ for $m>2$ is proved in Corollaries~\ref{cor:power} and~\ref{cor:other}. For the former, we only need Theorem~\ref{thm:cellSt} for a power of 2, in which case its proof simplifies drastically. For the latter, we only need Lemma~\ref{lem:m1m2} from Section~\ref{sec:GL}.

\section{preliminaries}\label{sec:prel}
Let $\bk$ be a field of characteristic $p>0$.

\subsection{Partitions and symmetric groups}

\subsubsection{}
We denote the set of all partitions by $\part$. We can decompose it as
\[\part\;=\; \bigsqcup_{n\in\mN}\part_n\;=\;\bigsqcup_{m\in\mN}\part(m),\]
where $\part_n$ contains all partitions of size $n$, hence all $\lambda\vdash n$, and $\part(m)$ contains all partitions of length precisely $m$, so $\ell(\lambda)=m$. We also write $\part(\le m)$ for the set of partitions of length at most $m$, meaning $\lambda_{m+1}=0$. Let $\lambda^t$ be the transpose of $\lambda\in\part$.
For $m,i\in\mZ_{>0}$, we fix the notation
\[\rho_m=(m-1,m-2,\ldots,2, 1)\in\part(m-1)\quad\mbox{and}\quad \omega_i=(1,\ldots, 1)\in \part(i).\]
Assuming the prime $p$ is clear from context, we denote the set of $p$-regular partitions as $\ppart$, and we similarly take its subsets $\ppart_n$ and $\ppart(m)$.

\subsubsection{} We let $S_n$ be the permutation group of the set $\{1,2,\ldots,n\}$ for $n\in\mZ_{>0}$. This gives natural inclusions $S_n< S_{n+1}$. More generally, we interpret $S_{m}\times S_n< S_{m+n}$ as the Young subgroup that stabilises the sets $\{1,\ldots,m\}$ and $\{m+1,\ldots,m+n\}$. We consider $S_\infty:=\cup_nS_n$.

As in \cite{BK, Kl-branching}, the simple modules of $\bk S_n$ are labelled as $D^\lambda$ for $\lambda\in\ppart_n$. 
 We denote by $e_\lambda$ an arbitrary primitive idempotent (unique up to conjugacy) corresponding to $D^\lambda$.

\subsection{Symmetric monoidal categories} See \cite{EGNO} for background on monoidal categories.

\subsubsection{}
We will use the term {\bf $\bk$-SM category} to refer to any $\bk$-linear additive, idempotent complete, symmetric monoidal category. We say that such $(\cC,\otimes,\unit)$ is {\bf integral} if $X\otimes Y=0$ for $X,Y\in\cC$ implies at least one of $X$ or $Y$ is zero. A {\bf tensor functor} between $\bk$-SM categories is a $\bk$-linear symmetric monoidal functor. For a finite-dimensional vector space $V$ and a $\bk$-SM category $\cC$ with $X\in\cC$ we have the object $V\otimes_{\bk}X$, which is isomorphic to $X^{\oplus \dim V}$, that represents the functor
\[\Hom_{\bk}(V,\Hom(X,-))\]
from $\cC$ to the category of vector spaces.

\subsubsection{}

We refer to \cite{Selecta} for some background on tensor ideals. A {\bf thick tensor ideal $\bI$} in a $\bk$-SM category $\cC$ is a collection of objects (closed under isomorphism) in $\cC$ such that $X\in\bI$ implies that $Y\otimes X\in \bI$ for all $Y\in\cC$, and furthermore $X_1\oplus X_2\in\bI$ if and only if $X_1$ and $X_2$ are both in $\bI$. A thick tensor ideal $\bI$ is {\bf prime} if it does not contain $\unit$ and if $X\otimes Y\in\bI$ implies that at least one of $X,Y$ is in $\bI$. 

A {\bf tensor ideal} $\cI$ in $\cC$ is an ideal in the additive sense (a system of subspaces $\cI(X,Y)\subset\Hom(X,Y)$ for all $X,Y\in \cC$ that is closed under composition on either side with arbitrary morphisms) such that $f\in \cI$ implies that $Z\otimes f\in\cI$ for all $Z\in\cC$. For a tensor ideal $\cI$, the quotient category $\cC/\cI$ is canonically a $\bk$-SM category again. A tensor ideal $\cI$ is {\bf prime} if $f\otimes g\in\cI$ implies $f\in\cI$ or $g\in\cI$ for morphisms $f,g$. We refer to \cite{CEO} for more details on prime tensor ideals.

The collection of identity morphisms contained in a tensor ideal $\cI$ defines a thick tensor ideal $\bI=\Ob(\cI)$, yielding a surjection from the set of tensor ideals onto the set of thick tensor ideals. A naive section of this surjection is given by associating to a thick tensor ideal $\bI$ the collection $\cI$ of morphisms that factor through objects in $\bI$. By slight abuse of notation, we will write $\cC/\bI$ for the quotient with respect to this tensor ideal. By definition, $\cC/\bI$ is integral if and only if $\bI$ is prime.

In order to match with conventions from \cite{Tprime}, we might sometimes ignore the zero ideal as a prime ideal, or consider the identity ideal to be prime.

\subsubsection{} A $\bk$-SM category $\cC$ is {\bf rigid} if every $X\in\cC$ has a monoidal dual $X^\ast$. We will often use the fact that, by the properties of the defining evaluation morphism $\ev_X:X^\ast\otimes X\to\unit$, the object $X$ is a direct summand of $X\otimes X^\ast\otimes X$ since $X\otimes \ev_X$ is split.

There is a unique maximal tensor ideal $\cM$ in a rigid $\bk$-SM category $\cC$ with $\End(\unit)=\bk$. Under suitable assumptions, see \cite{EO} and references therein, the quotient $\cC/\cM$ is semisimple abelian, and known as the {\bf semisimplification} $\overline{\cC}$ of $\cC$. An object in $\cC$ is called {\bf negligible} if it is zero in $\overline{\cC}$.

\subsubsection{}Denote by $\Sym^0$ the universal $\bk$-linear symmetric monoidal category on one generator. More concretely, its set of objects is given by $\mN$, the endomorphism algebra of $n$ is $\bk S_n$, there are no morphisms between different objects, tensor product is given by addition on objects and the inclusions $\bk(S_m\times S_n)\subset\bk S_{m+n} $ on morphisms. For clarity, we will denote the object $n$ by $E^{\otimes n}$. 

We let $\Sym$ be the Karoubi envelope of $\Sym^0$, which is thus a $\bk$-SM category. Its objects are direct sums of pairs $(E^{\otimes n},e)$ where $e\in\bk S_n$ is an idempotent. Then, for any $\bk$-SM category $\cC$ there is an equivalence of categories between the category of tensor functors $\Sym\to \cC$ and the underlying category of~$\cC$, given by $F\mapsto F(E)$.

\subsection{Tilting modules}\label{sec:prelTilt}
See \cite{Jantzen} for background on the representation theory of reductive groups. 

\subsubsection{}For a split reductive group $G$ over $\bk$, which will mostly be of type $A$, we consider a maximal torus and Borel subgroup $T\subset B^+\subset G$, leading to a weight lattice $X$ with dominant weights $X^+$. We let $\rho\in \mZ[1/2]\otimes X$ be the half-sum of positive roots. Following \cite{Jantzen}, for $r\in\mZ_{>0}$, we write
\[X_r\;=\;\{\lambda\in X \mid 0\le  \langle \lambda,\alpha^\vee\rangle <p^r\mbox{ for every simple positive root $\alpha$}.\}\]
For $V\in\Rep G$, we denote its $r$-th Frobenius twist by $V^{(r)}$. Inside the representation category $\Rep G$, we have the rigid $\bk$-SM category $\Tilt G$ of tilting modules, see \cite[Appendix E]{Jantzen}. The indecomposable tilting modules are labelled as $T(\lambda)$, $\lambda\in X^+$. 

\subsubsection{} If $G=\GL_m$ for some $m\in \mZ_{>0}$, we take the usual choice of $T\subset B^+$ with $B^+$ comprising upper triangular matrices and its negative Borel $B$ of lower triangular matrices. In this case $X\cong \mZ^m$, where $X^+$ corresponds to non-increasing sequences of integers. Hence we can identify $\part(\le m)$ with a subset of $X^+$. In particular $\omega_i$, for $i<m$ are the fundamental weights.

We always denote by $V$ the natural $m$-dimensional representation of $\GL_m$.
The category $\Tilt \GL_m$ is generated as a $\bk$-SM category by $\w^i V$ and its duals. The $\bk$-SM subcategory $\Tilt^\circ\GL_m$ generated by $V$ has as indecomposable objects precisely $T(\lambda)$, with $\lambda\in\ppart\subset X^+$, see \cite{BrK}. We also denote by $\Tilt^\bullet\GL_m$ the rigid $\bk$-SM subcategory of $\Tilt\GL_m$ generated by $V$. The $\bk$-SM subcategory $\Tilt^+\GL_m$ generated  by the $\w^i V$ contains $T(\lambda)$ if and only if $\lambda\in \part(\le m)\subset X^+$. We have inclusions that are strict, except the vertical ones for $p>m$,
\[\xymatrix{
\Tilt^\circ\GL_m\ar@{^{(}->}[rr] \ar@{^{(}->}[d]&&\Tilt^\bullet\GL_m\ar@{^{(}->}[d]\\
\Tilt^+\GL_m\ar@{^{(}->}[rr] &&\Tilt\GL_m.}\]

\subsubsection{}\label{sec:SymTilt}\label{def:Km} By Schur-Weyl duality, the universal functor 
\begin{equation}\label{eq:SymtoTilt}
\Sym\to \Tilt^\circ \GL_m,\quad E\mapsto V
\end{equation} is full and essentially surjective, so we can interpret $\Tilt^\circ \GL_m$ as the quotient of $\Sym$ by the tensor ideal generated by the skew symmetriser in $\bk S_{m+1}$, see \cite{DP}. Any non-zero tensor ideal of $\Sym$ contains the kernel of this functor for some $m$, see \cite[Lemma~4.10]{CEO}. It is well-known, see for instance~\cite{BrK}, that \eqref{eq:SymtoTilt} satisfies 
\begin{equation}\label{eq:simpletilt}(E^{\otimes n},e_\lambda)\,\mapsto\, \begin{cases}
T(\lambda),&\mbox{if}\quad\lambda\in \ppart(\le m),\\
0,&\mbox{otherwise.}\end{cases}\end{equation}

The above paragraph also shows that we have canonical tensor functors
\[\Tilt^\circ \GL_m\;\to\;\Tilt^\circ \GL_{m-1},\]
that send $T(\lambda)$ to $0$ if $\lambda\in \ppart(m)$ and to the module with same notation otherwise. We denote the kernel of the tensor functor as $\cK_m$ and its corresponding thick tensor ideal in  $\Tilt^\circ \GL_m$ by~$\bK_m$.

\subsubsection{} For $\SL_m$ viewed as a subgroup of $\GL_m$, restriction along $\SL_m\to \GL_m$ introduces a map on weight lattices that we interpret as sending $X\cong\mZ^m$ to its quotient with respect to the subgroup generated by $\omega_m$, which we denote by $\lambda\mapsto [\lambda]$. Restriction yields a tensor functor
\[\Tilt \GL_m\;\to\; \Tilt\SL_m,\quad T(\lambda)\mapsto T([\lambda]).\]

\begin{example} \label{classSL2}

The case $G=\SL_2$ will be especially relevant. In this case $X\cong\mZ$, and we denote the indecomposable tilting modules as $T_i$ for $i\in\mN$. The natural representation will be denoted by $U=T_1$.
The tensor ideals in $\Tilt \SL_2$ were classified in \cite{Selecta}, and they are in bijection with thick tensor ideals, so we only describe the latter. The non-zero proper thick tensor ideals form one chain
\[\cdots\;\subset\; \bI_3\;\subset\; \bI_2\;\subset\; \bI_1\;\subset\;\Tilt SL_2,\]
where $\bI_n$ contains the indecomposables $T_i$ for $i\ge p^n-1$. In other words, the most canonical generator of $\bI_n$ is the Steinberg module $St_n=T_{p^n-1}$.
It is well-known all these tensor ideals are prime, see for instance~\cite[Lemma~2.4.3(1)]{CEOq}.
\end{example}


\section{Inductive systems}\label{sec:indsys}
In this section, we reinterpret some results from \cite{BK, Tprime, Za} and make some additional connections with monoidal categories. We let $\bk$ be a field of characteristic $p>0$.

\subsection{Inductive systems and ideals}

We define inductive systems following \cite{Za}.
\begin{definition} 
 An {\bf inductive system} $\Phi$ comprises a non-empty subset $\Phi^n\subset\Irr S_n$ for each $n\in\mZ_{>0}$, such that $D\in\Phi^n$ if and only if $[\Res^{S_{n+1}}_{S_n}D':D]\not=0$ for some $D'\in \Phi^{n+1}$.
\end{definition}

There is an obvious notion of inclusion between inductive systems, leading to a unique maximal inductive system $\Phi$ with $\Phi^n=\Irr S_n$, which we denote simply by $\Irr S$.

For inductive systems $\Phi$ and $\Psi$ we obtain an inductive system $\Phi\otimes \Psi$, where $(\Phi\otimes\Psi)^n\subset \Irr S_n$ contains all simple constituents of modules $D_1\otimes D_2$, with $D_1\in \Phi^n$, $D_2\in \Psi^n$.

\subsubsection{}
Following \cite{Za} an inductive system $\Phi$ is {\bf indecomposable} if we cannot write it as $\Phi=\Phi_1\cup \Phi_2$ for two proper inductive subsystems $\Phi_i\subset\Phi$.
We will also use a stronger notion, introduced in \cite{Tprime}. An inductive system $\Phi$ is {\bf T-indecomposable} if for every $D_1\in \Phi^m$, $D_2\in\Phi^n$, there exists $D\in \Phi^{m+n}$ with
\[[\Res^{S_{m+n}}_{S_m\times S_n}D:D_1\boxtimes D_2]\;\not=\;0.\]
Obviously, any minimal inductive system is indecomposable, but one can show that it must even be T-indecomposable, see \cite[Corollary~2.1.9(2)]{Tprime}. Weakly T-prime ideals in $\bk S_\infty$ were introduced in \cite[Definition~2.1.3]{Tprime}, but for the current paper we can simply take Lemma~\ref{lem:bijections}(1) as definition.

\begin{lemma}\label{lem:bijections}
\begin{enumerate}
\item The assignment that takes an ideal $I<\bk S_\infty$ and associates the tensor  ideal $\cI$ in $\Sym$ determined by $\cI(E^{\otimes n},E^{\otimes n})=I\cap\bk S_n$ yields an inclusion preserving bijection between ideals in $\bk S_\infty$ and tensor ideals in $\Sym$ with the property that $f\in \cI$ if $f\otimes E\in\cI$. This bijection restricts to a bijection between weakly T-prime ideals in $\bk S_\infty$ and prime tensor ideals in $\Sym$.
\item The assignment that takes an inductive system $\Phi$ and associates the thick tensor ideal $\bI$ in $\Sym$ determined by $(E^{\otimes n},e_\lambda)\in \bI$ if and only if $D^\lambda\not\in \Phi^n$ yields an inclusion reversing bijection between inductive systems and proper thick tensor ideals in $\Sym$ with the property that $X\in \bI$ if $X\otimes E\in\bI$. It restricts to a bijection between T-indecomposable inductive systems and prime thick tensor ideals not containing $E$.
\end{enumerate}
\end{lemma}
\begin{proof}
An ideal $\cI$ in $\Sym$ is the same as the choice of an ideal $I_n=\cI(n,n)<\bk S_\infty$ for every $n$. On the other hand, the assignment that takes an ideal in $I<\bk S_\infty$ and associates the ideals $I\cap\bk S_n$ gives a bijection between the set of ideals in $\bk S_\infty$ and the set of families of ideals $I_n<\bk S_n$ for which $I_n=I_{n+1}\cap\bk S_n$, from which the conclusion (1) follows.
Part (2) follows similarly.
\end{proof}

The above lemma explains the connection between ideals in $\bk S_\infty$ and inductive systems from a tensor ideal perspective.
We refer to \cite{BK, Tprime, Za} for full details and only record the precise statement that we will need within the current paper. By $\Ann_RM<R$, we refer to the annihilator of a module $M$ over a ring $R$.

\begin{proposition}\label{prop:ideals}
For an inductive system $\Phi$, there is an ideal $I(\Phi)<\bk S_\infty$, where
\[I(\Phi)\cap \bk S_n\;=\;\bigcap_{m\ge n}\bigcap_{D\in \Phi^m} \Ann_{\bk S_n}\Res^{\bk S_m}_{\bk S_n}D.\]
This assignment yields a bijection between minimal inductive systems and maximal ideals in~$\bk S_\infty$.
\end{proposition}
\begin{proof}
This is \cite[Proposition~2.6(ii)]{BK}, with explicit description of $I(\Phi)$ from \cite[1.1.5]{Tprime} (or \cite[Lemma~2.8]{BK} for the semisimple case).
\end{proof}

\begin{remark}
By Lemma~\ref{lem:bijections}, we can alternatively prove Proposition~\ref{prop:ideals} as the statement that there is a bijection between maximal thick tensor ideals and maximal tensor ideals in $\Sym$, which follows from the principles in \cite[\S 2.3]{CEOq}.
\end{remark}

\subsection{Inductive systems from symmetric monoidal categories}\label{sec:monoidal}

By Lemma~\ref{lem:bijections}(2), inductive systems can be `realised' inside the (universal) $\bk$-SM category $\Sym$. We expand this idea to exploit later in the paper.

\subsubsection{} Consider an object $X$ in a $\bk$-SM category $\cC$. For $\lambda\in \ppart_n$, the object $\Sch_\lambda(X):=e_\lambda(X^{\otimes n})$, with $\bk S_n$ acting via the braiding morphisms, is a direct summand of $X^{\otimes n}$. Up to isomorphism, it is independent of the choice of $e_\lambda$. In particular, we get a decomposition
\[X^{\otimes n}\;\cong\;\bigoplus_{\lambda\in\ppart_n}D^\lambda\otimes_{\bk}\Sch_\lambda(X).\]
By construction we have for $\lambda\in \ppart_m,\mu\in\ppart_n$
\begin{equation}\label{eq:SW}\Sch_\lambda(X)\otimes\Sch_\mu(X)\;\cong\; \bigoplus_{\kappa\in\ppart_{m+n}}\Sch_\kappa(X)^{c^\kappa_{\lambda,\mu}},\quad\mbox{with}\quad c_{\lambda,\mu}^\kappa\:=\left[\Res^{S_{m+n}}_{S_m\times S_n}D^\kappa:D^\lambda\boxtimes D^\mu\right]\end{equation}
and, for $\nu\in\ppart_n$
\begin{equation}\label{eq:kro}
\Sch_\nu(X\otimes Y)\;\cong\; \bigoplus_{\lambda,\mu\in \ppart_n}\left(\Sch_\lambda(X)\otimes \Sch_\mu(Y)\right)^{g^\nu_{\lambda,\mu}},\quad\mbox{with}\quad g_{\lambda,\mu}^\nu\:=\left[D^\lambda\otimes D^\mu:D^\nu\right].
\end{equation}
\begin{definition}
For $X\in \cC$ as above, the set $\Phi^n_X\subset\Irr S_n$ comprises all $D^\lambda$ for which $\Sch_\lambda(X)\not=0$.
\end{definition}
As we will typically use the same symbol for an object $X\in\cC$ and the same object in a quotient category $\cC/\bI$ by a thick tensor ideal $\bI$, we also use the notation $\Phi_{X;\bI}$ for $\Phi_X$ when $X$ is to be considered in $\cC/\bI$.
The following properties follow from equations~\eqref{eq:SW} and~\eqref{eq:kro}, or by definition.

\begin{lemma}\label{lem:systemfromSM}
Let $\cC$ be a $\bk$-SM category with non-zero objects $X,Y$.
\begin{enumerate}
\item If $X\otimes -$ is faithful on objects, then $\Phi_X$ is an inductive system.
\item If $\cC$ is integral, then the inductive system $\Phi_X$ is T-indecomposable.
\item If $\cC$ is integral, then $\Phi_{X\otimes Y}=\Phi_X\otimes \Phi_Y$.
\item We have $\Phi_X\subset\Phi_{X\oplus Y}\supset \Phi_Y$.
\item If $F:\cC\to\cD$ is a faithful tensor functor, then $\Phi_X=\Phi_{F(X)}$.
\end{enumerate}
\end{lemma}

\begin{remark}
We could also consider the annihilator ideal of $X$ in $\bk S_\infty$ following \cite[\S 2.4 or \S 6.1]{Tprime}, then the inductive system associated to this ideal following \cite{Za} is $\Phi_X$.
\end{remark}

\begin{remark}
By Lemma~\ref{lem:bijections}, or alternatively universality of $\Sym$ and Lemma~\ref{lem:systemfromSM}(5), in order to obtain all inductive systems, it suffices to apply the above construction to $E$ in quotients of $\Sym$ by (thick) tensor ideals. \emph{However}, there are a number of reasons for expanding the definition as we did. Most notably applying Lemma~\ref{lem:systemfromSM}(3) or (4) requires working in a monoidal category where the relevant object factors into a tensor product or decomposes into a direct sum.
\end{remark}

\begin{example}\label{ex:BrK}
For $\cC=\Tilt^\circ\GL_m$, we have $\Sch_\lambda(V)\cong T(\lambda)$, see \eqref{eq:simpletilt}. In particular, \eqref{eq:SW} then produces \cite[Theorem~D(iii)]{BrK}:
\begin{equation}\label{eq:SWT}T(\lambda)\otimes T(\mu )\;\cong\; \bigoplus_{\kappa\in\ppart_{m+n}}T(\kappa)^{c^\kappa_{\lambda,\mu}},\quad\mbox{with}\quad c_{\lambda,\mu}^\kappa\:=[\Res^{S_{m+n}}_{S_m\times S_n}D^\kappa:D^\lambda\boxtimes D^\mu].\end{equation}
\end{example}

It will be convenient to have the following standard fact written out.
\begin{lemma}\label{lem:lambdam}
If $c^\kappa_{\lambda,\mu}\not=0$, for $\lambda,\mu\in\ppart(\le m)$,  then $\kappa_m\ge \lambda_m+\mu_m$
\end{lemma}
\begin{proof}
Since the highest weight in $T(\lambda)\otimes T(\mu)$ is $\lambda+\mu$, a non-zero $c^\kappa_{\lambda,\mu}$ implies $\lambda+\mu\ge \kappa$ in the dominance order, from which the inequality in the lemma follows.
\end{proof}

\begin{example}\label{ex:alpha}
Consider $m\in\mZ_{>0}$, and represent it in the form
\begin{equation}\label{eq:mdr}
m\;=\; (p-1)d+r,\quad\mbox{with}\quad 0<r\le p-1.\end{equation}
Then for any $\lambda\in \ppart(m)$, the $\GL_m$-tilting module $T(\lambda)$ is a direct summand of $T(\alpha)\otimes V^{\otimes N},$
for $\alpha=((d+1)^r,d^{p-1},(d-1)^{p-1},\ldots, 1^{p-1})$ and $N=|\lambda|-|\alpha|$, by \cite[Lemma~4.7]{BK} and \eqref{eq:SWT}.
\end{example}

\subsection{Length of inductive systems}

\begin{definition}
The {\bf length} of an inductive system $\Phi$ is
$$\ell(\Phi)\;=\;\sup\{\ell(\lambda)\mid  D^\lambda\in \Phi\}\;\in\;\mZ_{>0}\cup\{\infty\}.$$
\end{definition}

\subsubsection{}\label{sec:ml}By \cite[Proposition~5.2]{BK}, $\ell(\Phi)=\infty$ implies that $\Phi=\Irr S$. Note that this is just the shadow for thick tensor ideals of the claim on non-zero tensor ideals in $\Sym$ from \ref{def:Km}.
A related notion from \cite[\S 5]{BK} defines $m(\Phi)$, assuming $\ell(\Phi)<\infty$, as the maximal number $m\in\mZ_{>0}$ such that there are $D^\lambda\in\Phi$ with arbitrarily high $\lambda_m$. Hence $m(\Phi)\le\ell(\Phi)$.

\begin{remark}\label{rem:SysGL}
By \ref{sec:SymTilt}, the bijection in Lemma~\ref{lem:bijections}(2) restricts to a bijection between (T-indecomposable) inductive systems of length at most $m$ and (prime) proper thick tensor ideals $\bI$ in $\Tilt^{\circ}\GL_m$ with the property that $T\in\bI$ if $V\otimes T\in \bI$. This follows also from \eqref{eq:SWT}. 
\end{remark}

\begin{lemma}\label{lem:ell}
If $\Phi\not=\Irr S$ is T-indecomposable then $\ell(\Phi)=m(\Phi)$.
\end{lemma}
\begin{proof}
For a contradiction, let $D^\lambda\in\Phi$ be such that $\lambda_{\ell}$ is maximal.
For T-indecomposability, we need $c_{\lambda,\lambda}^\kappa\not=0$, in the notation of Example~\ref{ex:BrK}, for $\kappa$ of length at most $\ell$. By Lemma~\ref{lem:lambdam} this requires $\kappa_{\ell}\ge 2\lambda_{\ell}$, a contradiction.
\end{proof}

\begin{example}\label{ex:length2}
\begin{enumerate}
\item
For each $m\in\mZ_{>0}$, there is a unique maximal (T-indecomposable) inductive system $\Phi[m; \infty]$ of length $m$. It equals $\Phi_{V}$ for $V$ in $\Tilt^{\circ}\GL_m$ and comprises all $D^\lambda$ with $\lambda\in\ppart(\le m)$.
\item Recall $\Tilt SL_2$ and its classification of (prime) thick tensor ideals from Example~\ref{classSL2}.
Using Lemma~\ref{lem:systemfromSM}(2), for $j\in\mN$, we obtain T-indecomposable inductive systems
\[\Phi[2;  j]\;:=\;\Phi_{U;\bI_{j+1}}\;=\;\{D^{(\lambda_1,\lambda_2)}\mid  1\le \lambda_1-\lambda_2< p^{j+1}-1 \}, \]
where we exclude $j=0$ for $p=2$, since then $U\in\bI_1$.
As 2-row simple modules are well-understood, see \cite[Corollary~3.4]{KS}, we can deduce that, together with $\Phi[2; \infty]$, these are the only T-indecomposable inductive systems of length two, see Theorems~\ref{thm:3p} and~\ref{thm:32}.
\end{enumerate}
\end{example}

\begin{lemma}\label{lem:top}
Two indecomposable inductive systems $\Phi$ and $\Psi$ different from $\Irr S$ are equal if
\begin{enumerate}
\item they have the same length, say $m$, and
\item $D^\lambda\in \Phi$ if and only if $D^\lambda\in\Psi$ for all $\lambda\in\part^{\reg}(m)$.
\end{enumerate}
\end{lemma}
\begin{proof}
Consider an indecomposable inductive system $\Phi$ of length $m$. For each $n\in\mZ_{>0}$ we denote by $\Theta^n\subset\Phi^n$ the collection of all $D^\lambda\in\Phi^n$ for which $[\Res^{S_d}_{S_n}D^\mu:D^\lambda]=0$ for every $d\ge n$ and every $D^\mu\in\Phi^d$ with $\mu\in \ppart_d(m)$. While~$\Theta$ is not an inductive system, it does satisfy the property that each $D^\lambda\in \Theta^n$ is a constituent of the restriction of a simple in $\Theta^{n+1}$. We can thus define $\Phi_1\subset\Phi$ as the minimal inductive subsystem containing $\Theta$, i.e. the collection of simple constituents of restrictions of simples in $\Theta$. 

Next we denote by $\Phi_2\subset\Phi$ the complement of $\Theta$. In other words, $\Phi_2$ is the minimal inductive system that contains all length $m$ simples of $\Phi$. Clearly $\Phi=\Phi_1\cup\Phi_2$ and $\Phi_1\not=\Phi$ since it contains no length $m$ simples. Consequently $\Phi=\Phi_2$.

The claim now follows since for two indecomposable inductive systems $\Phi$ and $\Psi$ satisfying (1) and (2), we have $\Phi=\Phi_2=\Psi_2=\Psi$.
\end{proof}

\subsubsection{} Now we consider the minimal inductive system of a given length.
Fix $m\in\mZ_{>0}$ and represent it as in \eqref{eq:mdr}.
Following \cite{BK}, for $t\in\mN$, we set
\[\beta(m,t)\;=\; ((t+d)^{p-1},(t+d-1)^{p-1},\ldots,(t+1)^{p-1},t^r)\;\;\in\,\;\ppart(\le m).\]
By \cite[Lemma~4.10]{BK}, $D^{\beta(m,t)}$ is a constituent of the appropriate restriction of $D^{\beta(m,t+1)}$. 

\emph{Hence $D^{\beta(m,t)}$, for fixed $m$ and varying $t$ generate an inductive system, of length~$m$, which we denote by $\Phi[m]$.}

\begin{example}
\begin{enumerate}
\item The systems $\Phi[1],\ldots, \Phi[p-1]$ are the semisimple inductive systems by \cite{CS}. Moreover $\Phi[m]$, for $1\le m<p$, contains $D^\lambda$ for $\lambda\in\part(\le m)$ if and only if $\lambda_1-\lambda_m\le p-m$.
\item If $p>2$, the systems $\Phi[1],\ldots, \Phi[p-1]$ are the minimal inductive systems, by \cite{BK}.
\item If $p>2$, then $\Phi[2]=\Phi[2;0]$ from Example~\ref{ex:length2}.
\item If $p=2$, then $\Phi[2]=\Phi[2;1]$ from Example~\ref{ex:length2}, which is the spin inductive system, see \cite[\S 5 and \S 6]{Poly}  or \cite{GoK}. 
\item We can also identify $\Phi[p]$ for $p>2$, see Lemma~\ref{lem:Phip}.
\end{enumerate}
\end{example}

The main relevance of these inductive systems lies in the following observation of Baranov and Kleshchev (up to terminology). We will only need it for $p=2$, in which case, for completeness, we give a simple direct proof in the next subsection.
\begin{lemma}\label{lem:containsPhim}
\begin{enumerate}
\item Any T-indecomposable inductive system of length $m$ contains $\Phi[m]$.
\item Any inductive system $\Phi$ with $m=m(\Phi)$, contains $\Phi[m]$.
\end{enumerate}

\end{lemma}
\begin{proof}
In the proof of \cite[Theorem~5.7]{BK} it is shown that for any inductive system $\Phi\not=\Irr S$ it follows that `$m(\Phi)$ is admissible for $\Phi$' which is, by \cite[(9)]{BK}, equivalent to $\Phi[m(\Phi)]\subset\Phi$. This proves part (2), which includes part (1) as a special case by Lemma~\ref{lem:ell}.
\end{proof}

\subsection{Characteristic 2}For the remainder of the section we focus on $p=2$. We fix again $m\in\mZ_{>0}$.

\subsubsection{}\label{sec:2not}
It will often be convenient to have a symbol for the set of 2-regular partitions that are of length $m$ or $m-1$. We denote this set by $\ppart\langle m\rangle$, and it thus strictly contains $\ppart(m)$.

The expression for $\beta(m,t)$ now simplifies to
\[\beta(m,t)\;=\; (t+m-1,t+m-2,\ldots,t+1,t)=\rho_m+t\omega_m\;\in\;\ppart\langle m\rangle.\]
We also introduce the set $\gen\langle m\rangle$ that consists of all partitions $\lambda$ with $\ell(\lambda)\le m$ and $\lambda_i-\lambda_{i+1}\in\{1,2\}$
 for all $ 1\le i<m$.
In particular $\gen\langle m\rangle\subset\ppart\langle m\rangle$.

\begin{lemma}\label{lem:normal}
For any $\lambda\in\ppart(m)$, we have
$[\Res D^\lambda:D^{\beta(m,\lambda_m)}]\not=0,$
for the appropriate restriction.
\end{lemma}
\begin{proof}Using the terminology for removable nodes from for instance \cite[\S 3]{BK}, it is clear that if $\lambda\in\ppart(m)$ has no normal nodes for which removing them preserves 2-regularity in its first $m-1$ rows, then $\lambda=\beta(m,\lambda_m)$.
The result thus follows from applying \cite[Theorem~3.5]{BK} to remove normal nodes.
\end{proof}

\begin{proof}[Proof of Lemma~\ref{lem:containsPhim} for $p=2$]
For any $D^\lambda\in\Phi$ with $\lambda\in\ppart(m)$, Lemma~\ref{lem:normal} shows $D^{\beta(m,\lambda_m)}\in\Phi$ from which the conclusion follows.
\end{proof}


\begin{lemma}\label{lem:Phim} 
\begin{enumerate}
\item For any $\nu\in \part(\le m)$ and $\mu \in\gen\langle m\rangle$, for high enough $N\in\mN$ we have
\[[\Res D^{N\omega_m+\rho_m+2\nu}:D^{\mu+2\nu}]\not=0.\]

\item For any $\lambda\in\gen\langle m\rangle$  we have $D^\lambda\in \Phi[m]$.
\end{enumerate}
\end{lemma}
\begin{proof}
We can prove this by removing normal nodes via \cite[Theorem~3.5]{BK}. Adding an even partition $2\nu$ does not change which rows end in a normal node, it can only result in more of the normal nodes leading to regular partitions by removal and hence extra simple modules appearing after branching that we will not need. For (1), it is thus sufficient to prove the case $\nu=0$, which is equivalent to proving (2). It follows easily by removing good nodes via \cite[Theorem~3.5]{BK} that the simple for
\[(t+m-1,t+m-3,t+m-5,\ldots,t-m+1)\;=\;2\rho_m+(t-m+1)\omega_m\;=\;2\beta(m,\frac{t-m+1}{2})\]
is a constituent of the appropriate restriction of $D^{\beta(m,t)}$. Here we assumed that $t-m+1\ge 0$ (and is even for the last equation). We can then obtain all desired partitions by further removing good nodes. Indeed, one can do this by bringing the displayed partition to the desired shape from the bottom up, since the lower rows do not interfere with the goodness of higher nodes, proving (2).
\end{proof}

\begin{corollary}\label{cor:Phim}
The inductive system $\Phi[m]$ is T-indecomposable.
\end{corollary}
\begin{proof}
We apply to $\lambda=\beta(m,t)$ the special case of \eqref{eq:SWT}
\[\left[\Res^{S_{2n}}_{S_n\times S_n} D^{2\lambda}:D^\lambda\boxtimes D^\lambda\right]\;=\;1, \]and use the fact from Lemma~\ref{lem:Phim}(2) that $D^{2\beta(m,t)}\in \Phi[m]$. T-indecomposability follows since, by definition of $\Phi[m]$, every $D^\lambda\in\Phi[m]$ is a constituent of the restriction of some $D^{\beta(m,t)}$.
\end{proof}


\section{The cell of the Steinberg module in characteristic two}\label{sec:GL}

In this section, we let $\bk$ be a field of characteristic 2 and fix $m\in\mZ_{>1}$. We define a set $S\subset\mN$ by considering the $2$-adic expansion
\[m\;=\; \sum_{i\in S}2^i,\qquad \mbox{and set}\qquad \ell:=\max S.\]
Recall that we consider $\GL_m$ with fixed Borel subgroup and natural representation $V$, see~\S\ref{sec:prelTilt}. Our main result describes the cell of the first Steinberg module $T(\rho_m)$ in $\Tilt \GL_m$ and relates it to a subgroup $\SL_2^{\times\ell}<\GL_m$. Note that in characteristic $p\ge 2m-2$ (in fact at least $p\ge 2m-4$, see~\S\ref{sec:An}) this cell is understood. Indeed, by~\cite{An} the problem can then be reduced to understanding the cell of the trivial module. The latter is understood and connected to a principal $\SL_2<\GL_m$ via \cite{GK}. For $p=2$, we use the description of the cell of the trivial module in \cite{BE, BEEO}, but the principal $\SL_2$-subgroup, which does not exist for $m>2$ in characteristic $2$, will be replaced by $\SL_2^{\times \ell}$ which will play a similar, but much more intricate, role. We also require further work because we are ultimately interested in $\Tilt^\circ\GL_m$ rather than $\Tilt\GL_m$.

In \S\ref{sec:comnbinprep} we derive relations in the Grothendieck ring of $\SL_2^{\times \ell}$. In \S\ref{sec:SL2l}, we show how such combinatorics relates to tensor ideals in $\SL_2^{\times \ell}$. In \S\ref{sec:SL2config} we define the subgroup $\SL_2^{\times\ell}<\GL_m$. Section~\ref{sec:gen} contains more standard observations about tilting modules. In \S\ref{sec:GLmnegl} the above constructions are combined: we describe two tensor ideals in $\Tilt^\circ\GL_m$ via the restriction functor to $\Tilt\SL_2^{\times \ell}$ and our combinatorial tools for the latter. In \S\ref{sec:final3} we reformulate this in the specific form we need.

\subsection{Combinatorial preparation}\label{sec:comnbinprep}

\subsubsection{}\label{def:Lambda}We consider the (special) $\lambda$-ring $\mZ[x+x^{-1}]\subset\mZ[x,x^{-1}]$ with $\lambda^n(x+x^{-1})=\delta_{n,2}$ for $n>1$. In other words we work with the Grothendieck ring of the representation category of $\SL_2$ (in any characteristic). We consider the $\ell$-th tensor power $\Lambda$ of this $\lambda$-ring and denote the generators by $x_i+x_i^{-1}$. In other words, $\Lambda$ is the Grothendieck ring of $\SL_2^{\times\ell}.$ Inside the ring $\Lambda$ we denote by $\fm$ the ideal generated by the elements $(x^{2^{i-1}}_i+x^{-2^{i-1}}_i)=\psi^{2^{i-1}}(x_i+x_i^{-1})$, for $1\le i\le \ell$.

\begin{lemma}
For any $n\in \mN$ different from $0$ and $2^{\ell}$, we have $\lambda^n\left(\prod_{i=1}^\ell (x_i+x_i^{-1})\right)\in \fm$. On the other hand $\lambda^{n}\left(\prod_{i=1}^\ell (x_i+x_i^{-1})\right)=1$ for $n\in\{0,2^{\ell}\}$.
\end{lemma}
\begin{proof}
It is obvious that $\lambda^n$ produces $1$ for $n\in\{0,2^{\ell}\}$, and $0$ for $n>2^\ell$.
For each $0\le t<2^{\ell}$, we denote by $X_t$ the product $x_1^{\pm 1}x_2^{\pm 1}\cdots x_{\ell}^{\pm 1}$, where $x_i^{-1}$ appears if and only if $2^{i-1}$ appears in the 2-adic expansion of $t$. For example $X_0=\prod_i x_i$ and $X_{2^{\ell}-1}=X_0^{-1}$ and
\[\prod_{i=1}^\ell (x_i+x_i^{-1})\;=\; \sum_{t=0}^{2^{\ell-1}}X_t.\]

Application of $\lambda^n$, for $0<n<2^\ell$, on this element is the sum of all $n$-fold products without repetition of the $X_t$. The sum of all products for which precisely one of $X_{0}$, $X_{1}$ is a factor is in the ideal $(x_1+x_1^{-1})$, as they can be paired up in the obvious way. For the remaining products, we can consider the ones that contain precisely one of $X_{2}$ or $X_{3}$, and for the same reason their sum is in $(x_1+x_{1}^{-1})$. Continuing like this, we are left with the sum of all products that can be written as products of $X_{2i}X_{2i+1}$. Of these, we can now consider the sum of all products for which there is some $i$ such that precisely one of $X_{4i}X_{4i+1}$ or $X_{4i+2}X_{4i+3}$ occurs as a factor and observe that this sum is in $x_2^{2}+x_{2}^{-2}$. The remainder are the products of factors $X_{4i}X_{4i+1}X_{4i+2}X_{4i+3}$ which can be split up into a term belonging to $(x_3^4+x_3^{-4})$ and products of factors $\prod_{j=0}^7X_{8i+j}$, etc.
\end{proof}


\begin{corollary}\label{cor:lambda}
For $n\in\mN$ and the element $y=\sum_{i\in S}y_i$ of $\Lambda$, where $y_i=\prod_{j=1}^i(x_j+x_j^{-1})$, we have $\lambda^{n}(y)-1\in \fm$ if $n$ is a sum without repetition of terms $2^i$, $i\in S$, and
$\lambda^{n}(y)\in \fm$ otherwise.
\end{corollary}
\begin{proof}
Observe that $\lambda^n(y)$ is a sum of products $\lambda^{s_i}(y_i)$ with $\sum_is_i=n$. Modulo the ideal $\mathfrak{m}$ such a term
is $1$ when $s_i\in\{0,2^i\}$ for every $i\in S$, and $0$ otherwise. The result now follows.
\end{proof}

\subsection{Tensor ideals for products of $\SL_2$}\label{sec:SL2l}

We will only require a special case of the construction below, but write it out in more generality for future applications. In this section only $\ell$, rather than $S$, is relevant.

\subsubsection{}\label{def:Innn}
 In $\Tilt (\SL_2^{\times\ell})$, we label the indecomposable objects as $\otimes_a (T_{i_a})_{(a)}$, where $(T_{i})_{(a)}$ is  $T_i$ for the $a$-th copy of $\SL_2$ with trivial action of the other copies. 
For $T\in \Tilt (\SL_2^{\times\ell})$, its character $\ch T$ is an element of the ring $\Lambda$ from \ref{def:Lambda}. We also consider it as a function in $\ell$ variables. 

We fix $(n_1,\ldots, n_\ell)\in\mZ_{>0}^{\times\ell}$, and consider the thick tensor ideal $\bI_{n_1,\ldots, n_\ell}$ in $\Tilt SL_2^{\times\ell}$ generated by the modules $(St_{n_i})_{(i)}$ for $1\le i\le \ell$. Concretely, it has indecomposable objects 
\begin{equation}\label{eq:I}\bigotimes_{a=1}^{\ell} (T_{i_a})_{(a)},\quad\mbox{where}\quad i_c\ge 2^{n_c}-1\quad\mbox{for at least one}\quad 1\le c\le \ell.\end{equation}
It follows from the description in \eqref{eq:I} and the fact that the thick tensor ideals in $\Tilt \SL_2$ are prime that $\bI_{n_1,\ldots, n_\ell}$ is a prime thick tensor ideal.

Let $\zeta_n=\exp(\pi i/2^n)\in\mC$, $n\in\mZ_{>0}$, be a $2^{n+1}$-th primitive root of $1$. The following lemma (together with Lemma~\ref{lem:joep} below) generalises the case $\ell=1$ from \cite[Lemma~2.11]{Ne}.
\begin{lemma}\label{lem:joe2}
The following conditions are equivalent on $T\in \Tilt (\SL_2^{\times\ell})$:
\begin{enumerate}
\item $T\in \bI_{n_1,\ldots, n_\ell}$;
\item $\ch T$ is in the ideal of $\Lambda$ generated by the $\left(\frac{x_j^{2^{n_j}}-x_j^{-2^{n_j}}}{x_j-x_j^{-1}}\right)$, for $1\le j\le \ell$;
\item $\ch T$ is in the ideal of $\Lambda$ generated by the $x_j^{2^{n_j-1}}+x_j^{-2^{n_j-1}}$, for $1\le j\le \ell$;
\item $\ch T(\zeta_{n_1},\zeta_{n_2},\ldots,\zeta_{n_\ell})=0$.
\end{enumerate}
\end{lemma}
\begin{proof}
Since the character of $T_{2^n-1}$ is $\left(\frac{x^{2^{n}}-x^{-2^{n}}}{x-x^{-1}}\right)$, we have (1)$\Rightarrow$(2). We have (2)$\Rightarrow$(3) because of a clear inclusion of ideals. That (3)$\Rightarrow$(4) follows from $\zeta_n^{2^{n-1}}=-\zeta_n^{2^{n-1}}$.

The main task is showing that (4)$\Rightarrow$(1). By Donkin's tensor product theorem, see \cite{BEO, Do, Jantzen}, every non-trivial indecomposable tilting module for $\SL_2$ is of the form
\[T_{r_0}\otimes T_{r_1}^{(1)}\otimes\cdots \otimes T_{r_s}^{(s)}\]
for $1\le r_j\le 2$ and $s\in \mN$. Since the character of $T_1$ is $x+x^{-1}$ and that of $T_2$ is $(x+x^{-1})^2$, it follows that the evaluation of the character of any $\SL_2$-tilting module at any $\zeta_n$ is an element of $\mR_{\ge 0}$. Indeed, for the displayed equation, evaluation at $\zeta_n$ is 
\[\prod_{j=0}^s \left(2\cos\left(\pi/2^{n-j}\right)\right)^{r_j}.\]
This is strictly positive if $s+1<n$ and zero otherwise. Consequently,  $\ch T$ evaluated at $(\zeta_{n_1},\ldots,\zeta_{n_\ell})$ can only be zero if every indecomposable summand of $T$ satisfies this condition. So we can consider $T=\otimes_a (T_{i_a})_{(a)}$ and by the same reason, its character can only be zero if $\ch T_{i_a}(\zeta_{n_a})=0$ for some~$a$. By the above computation of the character with Donkin's tensor product formula, this implies $T_{i_a}\in \bI_{n_a}$, see also \cite[Lemma~2.11]{Ne}.
\end{proof}

\subsubsection{}\label{linkidealVer}In \cite{BE}, the abelian envelope $\Ver_{2^n}$ of $\Tilt \SL_2/\bI_n$ was constructed. In particular there is a (defining) tensor functor $\Tilt\SL_2\to \Ver_{2^n}$ with kernel $\bI_n$, we denote the image of $U$ in $\Ver_{2^n}$ by $L_1^{[n]}$ (hence $L_1^{[n]}=0$ if and only if $n=1$). As shown in \cite{BE} there are inclusions $\Ver_{2^n}\subset\Ver_{2^{n+1}}$, which allows us to consider the union $\Ver_{2^\infty}$. It then follows that we can equivalently define $\bI_{n_1,\ldots, n_\ell}$ as the kernel of the tensor functor
\[\Tilt SL_2^{\times \ell}\to \Ver_{2^\infty},\quad U_{(a)}\mapsto L_1^{[n_a]},\quad 1\le a \le\ell.\]

\subsection{The replacement of the principal $\SL_2$}\label{sec:SL2config}

\subsubsection{} Denote the standard embedding $\SL_2\hookrightarrow\GL_2$ by $\phi_2$. More generally, we define an inclusion
\[\phi_m:\,\SL_2^{\times \ell}\,\hookrightarrow\, \GL_{m}\]
iteratively. Concretely, we consider the embedding 
\[\phi_m':\,\SL_2\hookrightarrow\GL_m,\quad A\mapsto \begin{pmatrix}
A & 0 & \cdots & 0 \\
0 & A & \cdots & 0 \\
\vdots & \vdots & \ddots & \vdots \\
0 & 0 & \cdots & A
\end{pmatrix}\quad\mbox{or}\quad \begin{pmatrix}
A & 0 & \cdots & 0 \\
0 & A & \cdots & 0 \\
\vdots & \vdots & \ddots & \vdots \\
0 & 0 & \cdots & (1)
\end{pmatrix}\]
where the matrix on the right is a block matrix of $2\times 2$-blocks, if $m$ is even, with 1 row and column added if $m$ is odd. The image has a $\GL_{\lfloor m/2\rfloor}$-centraliser consisting of block matrices of the same division where every $2\times 2$-block is a scalar matrix. We denote that inclusion by $\phi''_{m}:\GL_{\lfloor m/2\rfloor}\hookrightarrow \GL_m$. We then define 
\[\phi_{m}(A_1,\ldots, A_{\ell})=\phi'_{m}(A_1)\cdot \phi''_{m}(\phi_{\lfloor m/2\rfloor}(A_2,\ldots, A_{\ell})).\]

\subsubsection{}It follows that the pull-back $X\to\mZ^{\times\ell}$ of $\phi_m$ on weights sends each $\epsilon_i$ to a sequence $(a_1,a_2,\ldots, a_{\ell})$, where $a_j\in\{-1,0,1\}$. For $W$ in $\Rep\GL_m$, we will denote by 
\[\ch_{\phi}W\;\in\; \Lambda\subset\mZ[x_1^{\pm 1},x_2^{\pm 1},\ldots, x_\ell^{\pm 1}]\]
its character after restriction along $\phi_m$. Here $\Lambda$ is the ring from \ref{def:Lambda}.

\begin{example}
For $m=4$, we get the embedding
\[\phi_4:\,\SL_2\times\SL_2\hookrightarrow \GL_4,\quad (A,\begin{pmatrix}
a&b\\
c&d
\end{pmatrix})\mapsto\begin{pmatrix}aA&bA\\
cA&dA\end{pmatrix}.\]
The pullback on weights is $\epsilon_1\mapsto (1,1)$, $\epsilon_2\mapsto (-1,1)$, $\epsilon_3\mapsto (1,-1)$ and $\epsilon_4\mapsto (-1,-1)$.
\end{example}

\begin{example}\label{ex:character}
The restriction of the natural $\GL_m$-representation $V$ under $\phi_m$ is given by the direct sum over $i\in S$ of $U_{(1)}\otimes U_{(2)}\otimes \cdots\otimes U_{(i)}$, to be interpreted as $\unit$ for $i=0$.
Hence
\[\ch_\phi V\;=\;\sum_{i\in S}\prod_{j=1}^{i}(x_j+x_j^{-1}).\]
For instance, if $m=2^\ell$, then $V$ restricts to $U_{(1)}\otimes \cdots \otimes U_{(\ell)}$.
\end{example}

\begin{remark}
We will use $\phi_{m}$ to pull back thick tensor ideals of $\Tilt\SL_2^{\times \ell}$ to $\GL_m$. This is in line with a conjectural classification of thick tensor ideals in $\Tilt\GL_m$ over fields of characteristic $p>m$ in \cite{AHR2} (in which case $\Tilt\GL_m=\Tilt^{\bullet}\GL_m$). Indeed, the latter proposes to classify thick tensor ideals in terms of orbits in the nilpotent cone $\mathcal{N}$, and the centralisers of the $\SL_2$-homomorphisms corresponding to these nilpotent elements (performed in an iterative manner, where one subsequently considers nilpotent orbits in these centralisers). In our rather degenerate characteristic 2 situation, the above considers the maximal nilpotent orbit for which there exists a corresponding $\SL_2$-embedding (in other words we focus on the restricted nullcone $\mathcal{N}_1$, as dictated also by support theory, see for example~\cite{CLNP}), the $\phi_m'$, and in line with \cite{AHR2} we augment by iteratively considering the centraliser. Hence, as will be confirmed below, we can think of $\phi_m$ as a replacement for a principal $\SL_2$-homomorphism. 
\end{remark}

\subsubsection{}\label{def:S}
Restriction along $\phi_{m}$ gives a tensor functor from
$\Tilt^{\circ}\GL_m$ to $\Tilt \SL_2^{\times\ell}$.
We define $\bS$  as the preimage in $\Tilt^{\circ}\GL_m$ of $\bI:=\bI_{2,3,\ldots,\ell+1}$ from \ref{def:Innn}, which is thus by construction a \emph{prime} thick tensor ideal. By \ref{linkidealVer} we can equivalently define $\bS$ as the kernel of the tensor functor
\[\Tilt^\circ \GL_m\to \Ver_{2^\infty},\quad V\mapsto X_m:=\bigoplus_{i\in S} L_1^{[2]}\otimes L_1^{[3]} \cdots\otimes L_{1}^{[i+1]},\]
where we note that by \cite{BE, BEO} the $|S|$ summands of $X_m$ above are simple in $\Ver_{2^\infty}$.

\subsection{Generalities about tilting modules}\label{sec:gen}

\begin{proposition}\label{prop:consol}
If $\mu\in\part(\le m)$ and $\lambda\in\ppart\langle m\rangle$, then $T(\mu)\otimes T(\lambda)$, so in particular $T(\lambda+\mu)$, is a direct summand of $V^{\otimes |\mu|}\otimes T(\lambda)$. All these representations thus belong to $\Tilt^\circ \GL_m$.
\end{proposition}
\begin{proof}
We claim that the canonical surjection
$$T(\rho_m)\otimes\left(V\otimes \w^{i-1}V\tto\w^{i}V\right) $$
is split. Indeed, using \cite[II.3.19(4)]{Jantzen},
\[T(\rho_m)\otimes L\;\cong\; \Indu^{\GL_m}_{B}(\bk_{\rho_m}\otimes L)\]
is a tilting module by \cite[II.4.5 and II.5.4(a)]{Jantzen}, for every simple constituent $L$ of $V\otimes\w^{i-1}V$. By the vanishing of extensions between tilting modules, see \cite[Corollary~E.2]{Jantzen}, the claim follows.

Now consider an arbitrary $\lambda\in\ppart\langle m\rangle$. By writing $\lambda=\kappa+\rho_m$, where $\kappa$ is thus a partition of length at most $m$, and viewing $T(\lambda)$ as a direct summand of $T(\kappa)\otimes T(\rho_m)$, we thus find that $T(\lambda)$ also splits the defining surjections of $\w^\bullet V$.

 This means that $T(\lambda)\otimes \w^i V$, by iteration, is a direct summand of $T(\lambda)\otimes V^{\otimes i}$. This result then immediately extends to $T(\lambda)\otimes \w^{\mu^t}V$ for arbitrary partitions $\mu$ of length at most $m$ and consequently to $T(\mu)\otimes T(\lambda)$.
\end{proof}

%

%
%
%

\begin{lemma}\label{lem:cellSt}
\begin{enumerate}
\item If a thick tensor ideal of $\Tilt^{\circ}\GL_m$ contains $T(\lambda)$ for some $\lambda\in \ppart\langle m\rangle$, then it contains $T(\mu+\lambda)$, for every $\mu\in\part(\le m)$.
\item  If a thick tensor ideal of $\Tilt^{\circ}\GL_m$  contains $T(\lambda+2\nu)$ for some $\lambda\in\gen\langle m\rangle$ and $\nu\in\part(\le m)$, then it contains $T(\rho_m+a\omega_m+2\nu)$ for some $a\in\mN$.
\end{enumerate}
\end{lemma}
\begin{proof}
Part (1) is an immediate consequence of Proposition~\ref{prop:consol}. Part (2) is a reformulation of Lemma~\ref{lem:Phim}(1), by \eqref{eq:SWT}.
\end{proof}

We record a special case of Donkin's tensor product theorem.
\begin{lemma}\label{lem:Donkin}
For any $\mu\in \gen\langle m\rangle$ and $\nu\in X^+$, the tensor product $T(\mu)\otimes T(\nu)^{(1)}$ is a tilting module, and hence contains $T(\mu+2\nu)$ as a direct summand. Moreover, for $a\in\mZ$,
\[T(a\omega_m+\rho_m+2\nu)\;=\;T(a\omega_m+\rho_m)\otimes T(\nu)^{(1)}.\]
\end{lemma}
\begin{proof}
If $\mu\in\gen\langle m\rangle$, then $\mu-\rho_m\in X_1$, so the first sentence follows from \cite[Lemma E.9]{Jantzen}. The second sentence also follows from that lemma, since $T(\rho_m)=L(\rho_m)$ is simple, so certainly indecomposable, as a representation of the Frobenius kernel $(\GL_m)_1$.
Indeed, we can consider it over $\SL_m$, which is simple over $(\SL_m)_1$ by \cite[II.3.18]{Jantzen}.
\end{proof}

\begin{lemma}\label{lem:Stnew}
If a thick tensor ideal in $\Tilt^\circ\GL_m $ contains $T(\rho_m+a\omega_m)\otimes T^{(1)}$ for some $a\in\mN$ and a non-negligible tilting module $T\in\Tilt^+\GL_m$, then it contains $T(\rho_m+b\omega_m)$ for some $b\in\mN$. 
\end{lemma}
\begin{proof}
Under the assumptions there exists $W\in\Tilt \GL_m$ such that $T\otimes W$ has  a direct summand $\unit$. This means that for high enough $N$, there exists $Z\in \Tilt^+\GL_m$ such that $T\otimes Z$ has a direct summand $T(N\omega_m)$. Hence $T(2\rho_m+(a+2N)\omega_m)$ is a direct summand of
\[T(\rho_m+a\omega_m)\otimes T^{(1)}\otimes T(\rho_m)\otimes Z^{(1)}\]
and thus in the thick tensor ideal. The conclusion now follows from Lemma~\ref{lem:cellSt}(2) for $\lambda=2\rho_m+(a+2N)\omega_m$ and $\nu=0$.
\end{proof}

\subsection{Negligible tilting modules revisited}\label{sec:GLmnegl}
The semisimplification of $\Tilt \GL_m$ in characteristic~2 was described in \cite[\S 8]{EO} and \cite{BEEO}. It included the equivalence between (1) and (2) in Theorem~\ref{thm:N} below. For us it will be important to establish the equivalent criteria (3) and (4).

\subsubsection{}We define the set $\pA\langle m\rangle\subset\part(\le m)$ of partitions that are of the form 
\[(a_1^{2^{i_1}},a_2^{2^{i_2}},\ldots, a_n^{2^{i_n}})\]
with $\{i_1,i_2,\ldots, i_n\}=S$ (as unordered sets) and $a_1\ge a_2\ge \cdots \ge a_n\ge 0$. Equivalently, $\lambda\in\pA\langle m\rangle$ if and only if
\[\lambda^t\;=\;\sum_{j\in S} 2^j \omega_{f(j)},\quad\mbox{for some}\quad f:S\to \mN.\]

\begin{example}
\begin{enumerate}
\item $\quad\pA\langle 2^\ell\rangle=\{a\omega_{2^{\ell}}\mid a\in\mN\}$
\item $\quad\pA\langle 3\rangle=\{a\omega_1+b\omega_3,c\omega_2+d\omega_3\mid a,b,c,d\in\mN\}$
\end{enumerate}
\end{example}
 
\begin{theorem}\label{thm:N}
The following conditions are equivalent on $T\in \Tilt^+\GL_m$.
\begin{enumerate}
\item $T$ is negligible in $\Tilt\GL_m$;
\item $T$ is a direct sum of modules $T(\mu)$ with $\mu\in\part(\le m)\backslash \pA\langle m\rangle$;
\item $\ch_\phi T\in \fm$;
\item $T(\rho_m)\otimes T^{(1)}\in \bS$, with $\bS$ from \ref{def:S}.
\end{enumerate}
\end{theorem}
\begin{proof}
That (1) and (2) are equivalent is proved in \cite{BEEO}. To show that (2) implies (3), it is sufficient to focus on indecomposable tilting modules. We prove the claim by induction on the Bruhat order $\le$ (which is the dominance order when restricted to partitions of the same size) on $\part(\le m)\subset X^+$. We write a partition $\lambda$ as
\[\lambda=(b_1^{r_1},b_2^{r_2},\cdots, b_s^{r_s})\]
with $b_i>b_{i+1}>0$. Now assume that $\lambda$ is not in $\pA\langle m\rangle$, so that one of the partial sums $r_1+r_2+\ldots+r_i$ does not equal a sum without repetitions of $2^j$, $j\in S$. Moreover, $T(\lambda)$ is a direct summand of
\[D:=\w^{\lambda^t}V=(\w^{r_1}V)^{\otimes b_1-b_2}\otimes (\w^{r_1+r_2}V)^{\otimes b_2-b_3}\otimes \cdots,\]
and all other direct summands of $D$ are of the form $T(\mu)$ with $\mu<\lambda$. By the above and equivalence of (1) and (2) it follows that one of the factors in the tensor product $D$ is negligible and hence $D$ negligible. Similarly, by Corollary~\ref{cor:lambda}, $\ch_\phi D\in \fm$.  Assume first that $\lambda$ is minimal under $\le$ with the property $\lambda\not\in\pA\langle m\rangle$. It then follows from equivalence of (1) and (2) that $D=T(\lambda)$ and thus indeed $\ch_\phi T(\lambda)\in \fm$, which takes care of the base case of the induction. For the induction step, it follows from the above that the difference of $\ch_\phi T(\lambda)$ and $\ch_\phi D$ is in $\fm$ and $\ch_\phi D$ is in $\fm$, concluding the proof of (2)$\Rightarrow $(3).

Now we prove that (3) implies (4). The condition $\ch_\phi T\in \fm$ means that $\ch_\phi T^{(1)}$ is in the ideal from Lemma~\ref{lem:joe2}(3) for the case $n_i=i+1$. Clearly $\ch_\phi(T(\rho_m)\otimes T^{(1)})$ is thus in the same ideal, and thus $T(\rho_m)\otimes T^{(1)}$ in $\bS$ by Lemma~\ref{lem:joe2}.

Finally, we prove that (4) implies (1).
Assume for a contradiction that $T$ is not negligible but that $T(\rho_m)\otimes T^{(1)}\in \bS$. By Lemma~\ref{lem:Stnew}, then $T(\rho_m+b\omega_m)\in\bS$ for some $b\in\mN$, so by Lemma~\ref{lem:joe2} $\ch_\phi T(\rho_m+b\omega_m)$ evaluates to zero on $(\zeta_2,\zeta_3,\ldots, \zeta_{\ell+1})$.
The (Weyl) character of $T(b\omega_m+\rho_m)$ is, up to a factor invisible by $\SL_m<\GL_m$, given by $\prod_{\alpha}(e^{\alpha/2}+e^{-\alpha/2})$, with $\alpha$ ranging over the positive roots of $\GL_m$. The restriction to $\SL_2^{\times\ell}$ is then given by a product over factors
\[x_1^{a_1}x_2^{a_2}\cdots x_{\ell}^{a_{\ell}}+x_1^{-a_1}x_2^{-a_2}\cdots x_{\ell}^{-a_{\ell}},\]
where each $a_i$ is in $\{1,1/2,0,-1/2,-1\}$, in a way that the outcome is in $\Lambda$, despite the local occurrences of $x_i^{\pm 1/2}$ (when $m$ is not a power of $2$). It is thus more convenient to rewrite the expression as a product of factors
\[y_1^{b_1}y_2^{b_2}\cdots y_{\ell}^{b_{\ell}}+y_1^{-b_1}y_2^{-b_2}\cdots y_{\ell}^{-b_{\ell}},\]
where $y_i^2=x_i$, and each $b_i$ is in $\{2,1,0,-1,-2\}$. Since we need to evaluate $x_i\mapsto \zeta_{i+1}$, we thus instead do $y_i\mapsto \zeta_{i+2}$, or equivalently $y_i\mapsto (\zeta_{\ell+2})^{2^{\ell-i}}$.
 For one of the factors to be zero, we need
\begin{equation}\label{eq:min1}-1\;=\;\zeta_3^{2b_1}\zeta_4^{2b_2}\cdots \zeta_{\ell+2}^{2b_{\ell}}\;=\;\zeta_{\ell+2}^{\sum_i b_i 2^{\ell-i+1}}\qquad\mbox{and thus}\qquad 1\;=\;\zeta_{\ell+2}^{\sum_i b_i 2^{\ell-i+2}}.\end{equation}
However, we have 
\[8-2^{\ell+3}\le\sum_{i=1}^{\ell} b_i 2^{\ell-i+2}\le 2^{\ell+3}-8\]
and $\zeta_{\ell+2}$ is a primitive $2^{\ell+3}$-th root of 1, so the only option for the right equation in \eqref{eq:min1} to be satisfied is having $\sum_{i=1}^{\ell} b_i 2^{\ell-i}=0$, which does not lead to the left equation in \eqref{eq:min1} being satisfied, a contradiction.
\end{proof}

\begin{lemma}\label{lem:StN}
If a thick tensor ideal in $\Tilt^\circ\GL_m $ contains $T(\mu+2\nu)$, for some $\mu\in\gen\langle m\rangle$ and $\nu\in \pA\langle m\rangle$, then it contains $T(\rho_m+b\omega_m)$ for some $b\in\mN$. 
\end{lemma}
\begin{proof}
If a thick tensor ideal contains $T(\mu+2\nu)$, with $\mu\in\gen\langle m\rangle$, then it must contain 
\[T(a\omega_m+\rho_m+2\nu)\;\cong\;T(\rho_m+a\omega_m)\otimes T(\nu)^{(1)},\] 
for some $a$, by Lemmata~\ref{lem:cellSt}(2) and~\ref{lem:Donkin}. If $\nu\in \pA\langle m\rangle$, then $T(\nu)$ is non-negligible by the result from \cite{BEEO} recalled in Theorem~\ref{thm:N}. The result thus follows from Lemma~\ref{lem:Stnew}.
\end{proof}

We do not need the following result for Theorem~A, but it leads to concrete new branching rules.
\begin{lemma}\label{lem:m1m2}
Choose a partitioning $S=S_1\sqcup S_2$ and set $m_i=\sum_{j\in S_i}2^j$, for $i\in\{1,2\}$, so that $m=m_1+m_2$. For any $r\in\mN$, the $\GL_m$-representation $T(2r\omega_m+2\rho_m)$ is a direct summand of 
\[T(2r\omega_{m_1}+\rho_m)\otimes T(2r\omega_{m_2}+\rho_m).\]
\end{lemma}
\begin{proof}
As observed in \cite[\S8]{EO}, the semisimplification of $\Tilt \GL_m$ is the category of $\mZ^S$-graded vector spaces. Hence, the inverse of $T(r\omega_{m_1})$, which is non-negligible by Theorem~\ref{thm:N}, in $\overline{\Tilt\GL_m}$ must be its dual $T(r(\omega_{m_2}-\omega_{m}))$. It follows that, inside $\Tilt\GL_m$,  $T(r\omega_m)$ is a direct summand of $T(r\omega_{m_1})\otimes T(r\omega_{m_2})$. The conclusion then follows, similarly to the general case in Lemma~\ref{lem:StN}, from Lemma~\ref{lem:Donkin}.
\end{proof}

\subsection{The cell of the Steinberg module}\label{sec:final3}

Now we come to the main result of this section, the description of which $T(\lambda)$ for $\lambda\in\ppart\langle m\rangle$ belong to $\bS$.

\begin{theorem}\label{thm:cellSt}
For $\lambda\in\ppart\langle m\rangle$, we have $T(\lambda)\not\in \bS$ if and only if $\lambda=\mu+2\nu$ for some $\mu\in\gen\langle m\rangle$ and $\nu\in \pA\langle m\rangle$.
\end{theorem}
\begin{proof}
By Theorem~\ref{thm:N} and Lemma~\ref{lem:Donkin}, we know $T(\rho_m+a\omega_m+2\nu)\not\in\bS$ for all $a\in\mN$ and $\nu\in \pA\langle m\rangle$. By Lemma~\ref{lem:StN}, this thus implies that $T(\mu+2\nu)\not\in\bS$ for all $\mu\in\gen\langle m\rangle$ and $\nu\in \pA\langle m\rangle$.

Conversely, for $\nu$ a partition of length at most $m$ that is {\bf not} in $\pA\langle m\rangle$, we know by Theorem~\ref{thm:N} and Lemma~\ref{lem:Donkin} that $T(\rho_m+2\nu)\in\bS$. It then follows from Lemma~\ref{lem:cellSt}(1) that $T(\mu+2\nu)\in\bS$ for all $\mu\in\gen\langle m\rangle$.

The conclusion now follows by observing that any $\lambda\in\ppart\langle m\rangle$ can be written uniquely as $\lambda=\mu+2\nu$ with $\mu\in\gen\langle m\rangle$ and $\nu\in \part(\le m)$.
\end{proof}

We conclude with a remark justifying the title of the section.

\begin{remark}
The results in this section demonstrate that the cell (in the sense of \cite{An,AHR2}) of $T(\rho_m)$ in $\Tilt\GL_m$ contains precisely the indecomposable tilting modules $T(\rho_m+\mu+2\nu)$ for $\mu\in X_1$ and $\nu\in \pA\langle m\rangle$. Indeed, that all these modules are contained in the cell of $T(\rho_m)$ follows almost immediately from Lemma~\ref{lem:StN}. Conversely, assume that $T:=T(\rho_m+\mu+2\nu)$ is in the cell of $T(\rho_m)$ for some $\mu\in X_1$ and $\nu\in\ppart(\le m)$. It follows that $T(\rho_m)$ is a direct summand of $T(\rho_m)\otimes T\otimes T^\ast$. By tensoring with $T(N\omega_m)$ for large enough $N$, it follows that $T(\rho_m+N\omega_m)$ is in the thick tensor ideal of $\Tilt^\circ\GL_m$ generated by $T(\rho_m+M\omega_m+\mu+2\nu)=T(\rho_m+M\omega_m+\mu)\otimes T(\nu)^{(1)}$, for appropriate $M\in\mN$ for which $\rho_m+M\omega_m+\mu$ is a partition (in $\gen\langle m\rangle$). It follows via Theorem~\ref{thm:N} that $T(\nu)$ cannot be negligible and thus $\nu\in\pA\langle m\rangle$.
\end{remark}


\section{Classification of maximal ideals}\label{sec:max}
Let $\bk$ be a field of characteristic $2$.

\subsection{Description of the candidate-minimal inductive systems}

\begin{theorem}\label{thm:FullPhim} Let $m\in\mZ_{>1}$.
\begin{enumerate}
\item We have $\Phi[m]=\Phi_{V;\bS}=\Phi_{X_m}$, for $\bS$ the prime thick tensor ideal in $\Tilt^\circ\GL_m$, and $X_m\in \Ver_{2^\infty}$, both defined in \ref{def:S}.
\item For $\lambda\in \ppart\langle m\rangle$, we have $D^\lambda\in \Phi[m]$ if and only if
\[\lambda\;=\;\mu+2\nu,\qquad\mbox{with}\quad \mu\in \gen\langle m\rangle \quad\mbox{and}\quad \nu\in\pA\langle m\rangle.\]
\end{enumerate}
\end{theorem}
\begin{proof}
We can rephrase Lemma~\ref{lem:StN} in terms of inductive systems, either via Lemma~\ref{lem:bijections}(2), or simply via \eqref{eq:SWT}. It then states that, for an inductive system $\Phi$, if $D^{\mu+2\nu}\not\in\Phi$ for some $\mu\in \gen\langle m\rangle$ and $\nu\in\pA\langle m\rangle$, then $D^{\rho_m+b\omega_m}\not\in \Phi$ for some $b\in\mN$. This implies that $\Phi[m]$ must contain all $D^\lambda$ with $\lambda$ as in (2).

We also know that $D^\lambda\in\Phi_{V;\bS}$ for $\lambda\in\ppart\langle m\rangle$ if and only if $\lambda$ is as in (2), by Theorem~\ref{thm:cellSt}, and the same conclusion thus follows also for $\Phi[m]$ by $\Phi[m]\subset \Phi_{V;\bS}$, see for instance Lemma~\ref{lem:containsPhim}, proving part (2). 

Part (1) now follows from the above and Lemma~\ref{lem:top}, where both $\Phi[m]$ and $\Phi_{V;\bS}$ are indecomposable, for instance by Corollary~\ref{cor:Phim} and Lemma~\ref{lem:systemfromSM}(2).
\end{proof}

\begin{corollary}\label{cor:power}
\begin{enumerate}
\item For $\ell\in\mZ_{>1}$, we have
\[\Phi[2^{\ell}]\;=\; \Phi[2;1]\otimes \Phi[2;2]\otimes \cdots\otimes \Phi[2;\ell].\]
\item For $\ell\in\mZ_{>1}$, we have $\Phi[2^{\ell-1}]\subset\Phi[2^{\ell}]$.
\end{enumerate}
\end{corollary}
\begin{proof}
That $\Phi_{V;\bS}$ in Theorem~\ref{thm:FullPhim}, for $m=2^\ell$, equals the right-hand side in part (1) follows from Example~\ref{ex:character} and Lemma~\ref{lem:systemfromSM}(3) and (5).

Part (2) is an immediate consequence of part (1), the fact that any simple $S_n$-representation $D$ is a simple constituent of $D^{\otimes 2}$ (as its own Frobenius twist), and the fact $\Phi[2;\ell-1]\subset \Phi[2;\ell]$.
\end{proof}

\begin{example}
The equality $\Phi[2;1]\otimes \Phi[2;2]=\Phi[4]$ also follows from the isomorphism
\[D^{(2r+2,2r)}\otimes D^{(2r+3,2r-1)}\;\cong\;D^{(r+2,r+1,r,r-1)},\]
for $r\in\mZ_{>0}$, since the three (families of) simple modules appearing are generators of the respective inductive systems. This isomorphism was conjectured in \cite{GK} and proved in \cite{GJ}. By the classification of simple tensor products of simple modules completed in \cite{Mo}, the other equalities in Corollary~\ref{cor:power}(1) do not follow from such isomorphisms.
\end{example}

\begin{corollary}\label{cor:other}
If $m$ is not a power of $2$, write $m=\sum_{j\in S}2^j$. We have $\Phi[m']\subset \Phi[m]$, with $m'=\sum_{j\in S'}2^j$, for any non-empty subset $S'\subset S$.
\end{corollary}
\begin{proof}
The system $\Phi_{V;\bS}$ in Theorem~\ref{thm:FullPhim} corresponds, by construction, to the inductive system of a direct sum $W$ of objects in $(\Tilt \SL_2^{\times \ell})/\bI$, see
 Example~\ref{ex:character}. Theorem~\ref{thm:FullPhim} identifies $\Phi[m]$ with $\Phi_{V;\bS}=\Phi_{W}$ and $\Phi[m']$ with $\Phi_{W'}$ for a direct summand $W'\subset W$, so that the inclusion follows from Lemma~\ref{lem:systemfromSM}(4) and (5).
\end{proof}

\begin{remark}
We can equivalently argue the above results via $\Ver_{2^\infty}$, as consequences of
\[X_m\;=\;\bigoplus_{i\in S} X_{2^i}\quad\mbox{and}\quad X_{2^\ell}=L_1^{[2]}\otimes \cdots \otimes L_1^{[\ell+1]},\qquad \mbox{with}\quad\Phi[2; j]=\Phi_{L_1^{[j+1]}}.\]

\end{remark}

The inclusions in Corollary~\ref{cor:other} could also be obtained without the full description in Theorem~\ref{thm:FullPhim}. Namely, we can reformulate Lemma~\ref{lem:m1m2}, by \eqref{eq:SWT}:
\begin{lemma}\label{lem:m1m2bis}
Choose a partitioning $S=S_1\sqcup S_2$ and set $m_i=\sum_{j\in S_i}2^j$, for $i\in\{1,2\}$, so that $m=m_1+m_2$. For any $r\in\mN$, 
\[[\Res D^{2r\omega_m+2\rho_m}:D^{2r\omega_{m_1}+\rho_m}\boxtimes D^{2r\omega_{m_2}+\rho_m}]\not=0.\]
\end{lemma}
Corollary~\ref{cor:other} then follows from Lemma~\ref{lem:m1m2bis} and the branching rules for removable nodes, see \cite[Lemma~4.12]{BK} for the precise result we need.

\subsection{Maximal ideals}

\begin{theorem}\label{thm:main}
\begin{enumerate}
\item The only minimal inductive systems are $\Phi[1]$ and $\Phi[2]$.
\item The only maximal ideals in $\bk S_\infty$ are the augmentation ideal $I[1]$ and the ideal $I[2]$ generated by the primitive idempotent of the trivial representation of $\bk S_3$.
\end{enumerate}
\end{theorem}
\begin{proof}
Let $\Phi$ be a minimal inductive system.
By \ref{sec:ml}, we can assume that $m=m(\Phi)$ is finite. Hence by Lemma~\ref{lem:containsPhim}, $\Phi=\Phi[m]$. It thus suffices to show that $\Phi[m]$ is not minimal when $m\not\in\{1,2\}$. Indeed, if $m$ is not a power of 2, this follows from Corollary~\ref{cor:other}. If $m>2$ is a power of 2, the result follows from Corollary~\ref{cor:power}(2). This proves part (1), and part (2) then follows from Proposition~\ref{prop:ideals} and the explicit description of the second maximal ideal in \cite[Lemma~6.2.5]{Poly}.
\end{proof}

Now we consider the alternating groups $A_n<S_n$ leading to the subgroup $A_\infty<S_\infty$ and subalgebra $\bk A_\infty\subset \bk S_\infty$.

\begin{theorem}\label{thm:Alt}
The algebra $\bk A_\infty$ has precisely two maximal ideals, which are the intersections of the maximal ideals of $\bk S_\infty$.   
\end{theorem}
\begin{proof}
There is again a bijection between maximal ideals in $\bk A_\infty$ and minimal inductive systems of $A_n$-representations, see~\cite{BK, Za}. We will denote $A_\infty$-inductive systems by $\Psi$ and ($S_\infty-$)inductive systems by $\Phi$.

It follows from applying the restriction functor $\Res$ and induction functor $\Indu$ between $\Rep A_n$ and $\Rep S_n$ that we can write $\ppart=\mathscr{A}\sqcup\mathscr{B}$ such that
\begin{equation}\label{eq:ResInd}\begin{cases}
\Res D^\lambda=E_+^\lambda\oplus E_-^\lambda,\quad\mbox{while}\quad \Indu E_+^\lambda\cong D^\lambda\cong \Indu E_-^\lambda, & \mbox{ if } \lambda\in \pA,\\
\Res D^\lambda =E^\lambda,\quad\mbox{while}\quad \Indu E^\lambda\;\, \mbox{is a self-extension of}\; D^\lambda,& \mbox{ if }\lambda\in\pB.
\end{cases}\end{equation}
Moreover, the modules $E^\lambda,E^\lambda_{\pm}$ are the simple $A_n$-representations, and they are all non-isomorphic.
In case $\bk$ is a splitting field for $A_n$, the set $\pA$ is described in \cite[Theorem~1.1]{Be}. If $\bk$ is not a splitting field, then $\pA$ is a proper subset of the one in \cite[Theorem~1.1]{Be}. The only important observation for us is that $\pA$ (for splitting fields and thus for any $\bk$) does not contain $(4s+2,4s)$, for $s\in \mN$.

For an $A_\infty$-inductive system, it is clear (without \eqref{eq:ResInd}, but using the Mackey formula) that $\Indu\Psi$, which we define as the system of all simple constituents of the modules induced from the simples in $\Psi$, is an inductive system. Similarly, from an inductive system $\Phi$, we obtain an $A_\infty$-inductive system $\Res\Phi$. By \eqref{eq:ResInd}, $\Res\Indu\Psi$ is the union of $\Psi$ with all $E_{\pm}^\lambda$ for which $E_{\mp}^\lambda$ as in $\Psi$. 

Now let $\Psi$ be an $A_\infty$-inductive system. By Theorem~\ref{thm:main}(1) and the above, $\Res\Indu\Psi$, and therefore also $\Psi$, contains either all the trivial representations or $E^{(4s+2,4s)}$ for $s\in \mN$. The latter case then easily implies that $\Psi$ contains all $E^{(a,b)}$ and $E^{(a,b)}_{\pm}$ (whichever exist for given $a$ and $b$) for all $1\le a-b\le 2$. We thus obtain precisely two minimal systems, which are $\Res \Phi[1]$ and $\Res \Phi[2]$. The corresponding ideals are the annihilators of those systems as in Proposition~\ref{prop:ideals}, which are the intersections of $\bk A_\infty$ with $I[1]$ and $I[2]$.
\end{proof}

\begin{remark}
In \cite{BK} it was proved that the number of maximal ideals in $\bk A_\infty$ for $p>2$ is $(p-1)/2$, so the discrepancy between $p=2$ and $p>2$ is more significant for $A_\infty$ than for $S_\infty$.
\end{remark}

\subsection{Hecke algebras}

In this section we let $\mF$ be a field of arbitrary characteristic $p\ge 0$. We pick $q\in \mF^\times$ for which
\[1+q+q^2+\cdots+q^{e-1}\;=\;0,\]
for some $e\in\mZ_{>1}$, which we henceforth fix to be minimal.
Hence either $q=1$ and $e=p>0$, or $q\not=1$ and $q$ is a primitive $e$-th root of $1$.

\subsubsection{}We let $H_n$ be the corresponding Hecke algebra at $q$ over $\mF$ as in \cite{Br}. For example $H_n=\mF S_n$ if $q=1$. We can once again consider a limit $H_\infty=\varinjlim H_n$. The simple modules of $H_n$ can be labelled by $e$-regular partitions of $n$, and the analogue of Kleshchev's branching rules were proved by Brundan in \cite[Theorem~2.5]{Br}. Also the Mullineux involution extends by \cite{Br}, so that one can copy the proof of \cite{BK} verbatim to obtain:

\begin{theorem}\label{thm:Hecke}
If $e>2$, there are precisely $e-1$ maximal ideals in $H_\infty$.
\end{theorem}
One can again verify that these maximal ideals come from the $e-1$ Markov traces on $H_\infty$.
However, the methods in the current paper (for $e=2$), which crucially rely on the Hopf algebra structure of $\mF S_\infty$ do not naively extend to $H_\infty$. Nonetheless, we do know:
\begin{proposition}\label{prop:Hecke}
If $\mathrm{char}(\mF)=0$ and $e=2$, there is precisely one maximal ideal in $H_\infty$.
\end{proposition}
\begin{proof}
Again, the methods from \cite{BK} allow us to conclude that the only candidates for minimal inductive systems are given by (the immediate analogue of) $\Phi[m]$. However, by \cite[Theorem~1.4]{JM}, the simple $H_n$-module $D^{\rho_m+a\omega_m}$ is a Specht module, which implies that its restriction to $H_{a+m-1}$ has a trivial module as simple constituent. Since $a$ can be taken to be arbitrarily high it follows that $\Phi[1]\subset \Phi[m]$, so there is only one minimal inductive system.  By the same reasoning as for symmetric group algebras this implies there is only one maximal ideal.
\end{proof}
The above proposition means that, contrary to symmetric group algebras in positive characteristic, the maximal ideals for Hecke algebras in characteristic zero behave homogeneously for all roots of $1$, and the maximal ideals are always given by the radicals of the Markov traces from \cite{We}.

A proof of the following conjecture would generalise Theorem~\ref{thm:main}(2) and complete this story for Hecke algebras. Note that one can show that there are at least two maximal ideals here.
\begin{conjecture}\label{conj:Hecke}
If $\mathrm{char}(\mF)>0$ and $e=2$, there are precisely two maximal ideals in $H_\infty$.
\end{conjecture}


\section{Categorical dimension of ideals and inductive systems of length two and three}\label{sec:allp}

Let $\bk$ be a field of characteristic $p>0$. In this section, we connect tensor ideals in $\Tilt^\circ \GL_m$ with the better developed theory of tensor ideals in the \emph{rigid} monoidal categories $\Tilt \SL_m$ and $\Tilt^\bullet\GL_m$. As a consequence of the first connection, we classify T-indecomposable inductive systems of length two and three. As a consequence of the latter, we prove \cite[Conjecture~6.3.4]{Tprime}.

\subsection{Reduction to the special linear group}

We already know that we can connect inductive systems with thick tensor ideals in $\Tilt^{\circ}\GL_m$. In this section we show how for T-indecomposable inductive systems we can instead work with $\Tilt \SL_m$. The latter is rigid, which provides more tools to deal with its tensor ideals, see for instance~\cite{Selecta}, and indeed its tensor ideal structure has received a lot of attention already, see \cite{An, AHR2, CEOq} and references therein.

\begin{theorem}\label{thm:SLGL}
\begin{enumerate}
\item Taking inverse images under $\Tilt^{\circ }\GL_m\to\Tilt^{\circ}\GL_{m-1}$ yields a bijection between (prime) thick tensor ideals in $\Tilt^{\circ}\GL_m$ that contain $\bK_m$ and (prime) thick tensor ideals in $\Tilt^{\circ}\GL_{m-1}$.
\item Every prime thick tensor ideal in $\Tilt^{\circ}\GL_m$ that does not contain $\bK_m$ is the preimage of a thick tensor ideal under $\Tilt^{\circ}\GL_m\to\Tilt \SL_m$.
\item  If $m< p$, then taking preimages as in (2) yields a bijection between prime thick tensor ideals in $\Tilt^\circ\GL_m$ and prime thick tensor ideals in $\Tilt \SL_m$.
\end{enumerate}

\end{theorem}

Before proceeding to the proof of the theorem, we spell out our main application. Denote by $V$ the natural representation in $\Tilt \SL_m$, and consider for any thick tensor ideal $\bI$ in $\Tilt\SL_m$
\[\Phi_{V;\bI}\,=\,\{D^\lambda\mid \lambda\in \ppart(\le m)\quad\mbox{with}\quad T([\lambda])\not\in\bI\}.\]

\begin{corollary}\label{cor:inducSL}
Every T-indecomposable inductive system of length $m$ is of the form $\Phi_{V;\bI}$
for $\bI$ a (prime, in case $m< p$) thick tensor ideal in $\Tilt \SL_m$. Conversely, if $\bI$ is a prime thick tensor ideal in $\Tilt \SL_m$ not containing $V$, then $\Phi_{V;\bI}$ is a T-indecomposable inductive system. 
\end{corollary}
\begin{proof}
This follows from combining Remark~\ref{rem:SysGL} with Theorem~\ref{thm:SLGL}.
\end{proof}

\begin{remark}
Even prime thick tensor ideals in $\Tilt^\circ\GL_m$ that do contain $\bK_m$ can be in the preimage of prime thick tensor ideals of $\Tilt\SL_m$. For example, if $p=2$ then the kernel of $\Tilt^\circ\GL_3\to\Tilt^\circ\GL_1$ equals the preimage of the maximal thick tensor ideal in $\Tilt\SL_3$, explaining why the latter does not play a role in Theorem~\ref{thm:32} below. 
\end{remark}

Now we start the proof of Theorem~\ref{thm:SLGL}. Recall the thick tensor ideal $\bK_m$ from \ref{def:Km} and $\alpha\in\ppart(m)$ from Example~\ref{ex:alpha}. We will freely use the observation, see Lemma~\ref{lem:lambdam}, that any direct summand $T(\kappa)$ of $T(\lambda)\otimes T(\alpha)$, with $\lambda\in\ppart(\le m)$, satisfies $\kappa\in \ppart(m)$.

\begin{proposition}\label{prop:SLGL}
Let $\bI$ be a prime thick tensor ideal in $\Tilt^\circ\GL_m$ that does not contain $\bK_m$, meaning there is some $\mu\in\part^{\reg}(m)$ with $T(\mu)\not\in\bI$.
\begin{enumerate}
\item For any $\lambda \in\part^{\reg}( m)$, we have $T(\lambda)\in\bI$ if and only if $T(\lambda-\omega_m)\in\bI$.
\item For any $\lambda\in\ppart(\le m)$ and $N\in\mN$, we have $T(\lambda)\in \bI$ if and only if $T(\lambda)\otimes T(\alpha+N\omega_m)\in\bI$.
\end{enumerate}
\end{proposition}
\begin{proof}
For part (1), assume that $T(\lambda-\omega_m)\not\in\bI$. Then with $\mu$ as in the lemma, since $\bI$ is prime, 
\[T(\lambda)\otimes T(\mu-\omega_m)\;=\; T(\lambda-\omega_m)\otimes T(\mu)\;\not\in\;\bI.\]
So clearly $T(\lambda)\not\in\bI$. Similarly, assume that $T(\lambda)\not\in\bI$, then looking at $T(\lambda-\omega_m)\otimes T(\mu+\omega_m)$
shows $T(\lambda-\omega_m)\not\in\bI$. This proves part (1).

Now we prove part (2) for the special case $N=0$.
By primeness of $\bI$ it suffices to observe $T(\alpha)\not\in\bI$. The latter follows from the fact that $T(\alpha)$ generates $\bK_m$ as a thick tensor ideal by Example~\ref{ex:alpha}.

Finally, for part (2) with $N>0$, we can observe that $T(\lambda)\otimes T(\alpha)$ is a direct sum of $T(\kappa)$ with $\kappa\in\ppart(m)$, and that by part (1) we have that $T(\kappa)\in\bI$ if and only if $T(\kappa+N\omega_m)\in \bI$. Hence $T(\lambda)\otimes T(\alpha)$ is in $\bI$ if and only if $T(\lambda)\otimes T(\alpha+N\omega_m)$ is in $\bI$, which allows us to reduce to the previous case $N=0$.
\end{proof}

\begin{proof}[Proof of Theorem~\ref{thm:SLGL}]
Since the functor to $\Tilt^{\circ}\GL_{m-1}$ is essentially surjective, part (1) is immediate.
We thus let $\bI$ be a prime thick tensor ideal in $\Tilt^{\circ}\GL_{m}$ that does not contain $\bK_m$. Our remaining task for (2) is to prove $\bI$ is the preimage of a thick tensor ideal under 
\[\Res:\Tilt^\circ\GL_m\to\Tilt\SL_m,\quad T(\lambda)\mapsto T([\lambda]).\] Let $\bJ$ be the minimal thick tensor ideal in $\Tilt\SL_m$ that contains every module in the image of~$\bI$. Clearly $\bI\subset\Res^{-1}(\bJ)$, and we will prove it is an equality. 

Consider $T(\lambda)$ with $\lambda\in \ppart(\le m)$ for which $T([\lambda])\in\bJ$. It easily follows that there exists $T\in\Tilt\GL_m$ and $T(\mu)\in \bI$ (so $\mu\in\ppart$) such that $T(\lambda)$ is a direct summand of $T\otimes T(\mu)$. Hence $T(\lambda)$ will split the evaluation morphism of $T(\mu)$, so that $T(\lambda)$ is a direct summand of $T(\lambda)\otimes T(\mu)\otimes T(\mu)^\ast$. It then follows further that, for any  $N\in \mN$, the module $T(\lambda)\otimes T(\alpha+N\omega_m)$ is a direct summand of
\[T(\lambda)\otimes T(\mu)\otimes T(\alpha)^{\otimes 2}\otimes (\w^m V)^{\otimes N} \otimes (T(\alpha)\otimes T(\mu))^\ast.\]
Now $T(\alpha)\otimes T(\mu)$ is a direct sum of tilting modules $T(\nu)$ with $\nu\in \ppart(m)$, and the highest weight of $T(\nu)^\ast$ is $(-\nu_{m},-\nu_{m-1},\ldots,-\nu_1)$. Hence, for high enough $N$, $(N-\nu_{m},N-\nu_{m-1},\ldots,N-\nu_1)$
will be a ($p$-regular) partition, which shows that for suitable $N$, the displayed equation is a tensor product of $T(\mu)$ with a tilting module in $\Tilt^\circ \GL_m$, showing that $T(\lambda)\otimes T(\alpha+N\omega_m)$ is in $\bI$. The conclusion follows from Proposition~\ref{prop:SLGL}(2).

Finally, for part (3), assume that $m< p$, so that $\Tilt^\circ \GL_m=\Tilt^+\GL_m$. It now follows easily that for any dominant $\SL_m$-weight, which we can always write as $[\lambda]$ for $\lambda\in\ppart(\le m)=\part(\le m)$, we have $T(\lambda)\in\bI$ if and only if $T([\lambda])\in\bJ$. Note that this statement makes sense thanks to Proposition~\ref{prop:SLGL}(1). It thus follows that $\bJ$ is also prime and that the procedures $\bI\leadsto \bJ$ and $\bJ\leadsto \Res^{-1}(\bJ)$ are now two-sided inverse, leading to the bijection.
 \end{proof}

\subsection{Andersen's Lemma}\label{sec:An}
The result of this section, Lemma~\ref{lem:An}, will only be used for $\SL_3$, but since the method works much more generally, we briefly work in the set-up, where $G$ is a reductive group, with only the assumption that $(p-1)\rho$ is a weight, so for instance $p>2$ or $G=\SL_m$.

In \cite[Lemma~13]{An}, Andersen proved a powerful result vastly reducing the study of the thick tensor ideals in $\Tilt G$, under the assumption that $p\ge 2h-2$, with $h$ the Coxeter number of $G$. This does not quite cover all cases we need. In \cite{An}, the above condition is used in two ways. The point of the current section is to demonstrate that one of these, injectivity of the Steinberg module in a certain subcategory, is not required. This will allow us to use the result in all cases we need. The remaining hypothesis is:

\subsubsection{Hypothesis}\label{hypo} Fix some $r\in\mZ_{>0}$.  For $\mu\in (p^r-1)\rho+X_r$ and $\nu\in X^+$, the tilting module $T(\mu)\otimes T(\nu)^{(r)}$ from \cite[Lemma~E.9]{Jantzen} is indecomposable, {\it i.e.}
\[T(\mu+p^r\nu)\;\cong\; T(\mu)\otimes T(\nu)^{(r)}.\]
By \cite{BNPS} this is satisfied whenever $p\ge 2h-4$.

\begin{lemma}\label{lem:An}
Assume Hypothesis~\ref{hypo} is valid for $r>0$. For $\lambda,\mu\in  (p^r-1)\rho+X^+$, written as $\lambda=\lambda_0+p^r\lambda_1$ with $\lambda_0\in (p^r-1)\rho+X_r$ and $\lambda_1\in X^+$ (and similarly for $\mu$), we have that $T(\mu)$ is in the tensor ideal generated by $T(\lambda)$ if and only if $T(\mu_1)$ is in the tensor ideal generated by $T(\lambda_1)$.
\end{lemma}
\begin{proof}
We start by proving that if $T(\mu)$ is in the tensor ideal generated by $T(\lambda)$ then $T(\mu_1)$ is in the tensor ideal generated by $T(\lambda_1)$. By \cite[Lemma~E.8]{Jantzen}, every indecomposable direct summand of the tensor product $T(\lambda_0)\otimes T$ for a tilting module $T$ must be of the form $T(\nu)$ for $\nu\in  (p^r-1)\rho+X^+$. Consequently, we can write $T(\lambda)\otimes T$ as a direct sum of modules
\[T(\nu_0)\otimes (T(\lambda_1)\otimes T(\nu_1))^{(r)}\]
with $\nu_0\in(p^r-1)\rho+X_r$ and $\nu_1\in X^+$. By \ref{hypo}, the only way that we can obtain $T(\mu)$ as a direct summand is if $\nu_0=\mu_0$ and if $T(\mu_1)$ is in the tensor ideal generated by $T(\lambda_1)$. This thus takes care of one direction of the lemma.

Next we claim that $T(\lambda_0)$ generates the same tensor ideal as $T((p^r-1)\rho)$. We use highest weight consideration to obtain the left commutative triangle in the following diagram in $\Tilt G$:
\[\xymatrix{
T(\lambda_0)\ar[rr]\ar@{=}[drr] && T((p^r-1)\rho)\otimes T(\lambda_0-(p^r-1)\rho)\ar[d]\ar@{-->}[rr]&&T(\kappa)\otimes T(\lambda_0-(p^r-1)\rho)\ar@{-->}[lld] \\
&& T(\lambda_0).
}\] 
In particular, via adjunction we obtain a non-zero morphism
\[T((p^r-1)\rho)\to T(\lambda_0)\otimes T(\lambda_0-(p^r-1)\rho)^\ast,\]
which is a monomorphism  since $T((p^r-1)\rho)$ is simple. There is a direct summand $T(\kappa)$, with $\kappa\ge (p^r-1)\rho$ potentially being an equality, of the right-hand side, in the block of $T((p^r-1)\rho)$, through which the inclusion factors. By adjunction, we can complete the above diagram with the dashed arrows. In particular we find that $T(\lambda_0)$ and $T(\kappa)$ actually generate the same ideal.

By the description of the blocks of $\Rep G$ in \cite[Theorem~5.8]{DonkinBlocks}, we have $\kappa=p^rw(\rho)-\rho+p^{r+1}\gamma$, for $w$ an element of the (finite) Weyl group and $\gamma\in X$ in the root lattice. Hence $\kappa-(p^r-1)\rho \in p^r X^+$, so $\kappa_0=(p^r-1)\rho$. This means that either $\kappa=(p^r-1)\rho$, in which case we are done, or an application of the first paragraph (and the conclusion of the second paragraph), implies that $\unit$ is in the tensor ideal generated by $T(\kappa_1)$. In other words, $\unit$ is a direct summand of $T(\kappa_1)\otimes D$ for a tilting module $D$. Subsequently $T((p^r-1)\rho)$ is a direct summand of 
\[T(\kappa)\otimes T((p^r-1)\rho)^\ast\otimes T((p^r-1)\rho)\otimes D^{(r)}\] 
and hence also generates the same tensor ideal as $T(\kappa)$, concluding the proof of the claim.

As a consequence of the claim proved above, also $T(\lambda)$ and $T((p^r-1)\rho)\otimes T(\lambda_1)^{(r)}$ generate the same tensor ideal as each other. We can thus henceforth assume that $\lambda_0=\mu_0=(p^r-1)\rho$.
It then follows instantly that $T(\mu)$ is in the tensor ideal generated by $T(\lambda)$ when $T(\mu_1)$ is in the tensor ideal generated by $T(\lambda_1)$, concluding the proof of the other direction of the lemma. 
\end{proof}

\begin{remark}
Under the assumption $p\ge 2h-2$ the proof simplifies, since then there are no other tilting modules in the block of $T((p^r-1)\rho)$ with weight inside $(p^r-1)\rho+X_r+X_r$, so in the proof of Lemma~\ref{lem:An} we know $\kappa=(p^r-1)\rho$ without any effort.
Interestingly, for the case $G=\SL_3$ and $p=2$ needed below, one can still apply this kind of proof, since the zero weight being singular again removes any problematic weights from the block of $T((p^r-1)\rho)$. However, for $G=\SL_3$ and $p=3$, we really need an argument as in the proof of Lemma~\ref{lem:An} above.
\end{remark}

\subsection{Thick tensor ideals for $\SL_3$}
Thick tensor ideals in $\Tilt \SL_3$ for $p\ge 5$ were classified in \cite{An}, see Proposition~\ref{prop:AnOdd} below. The classification for $p=2$ deviates from that:
\begin{proposition}\label{prop:An}
The non-zero proper thick tensor ideals in $\Tilt \SL_3$ for $p=2$ form one countable chain
$$\cdots\;\subset\; \bI_3\;\subset\; \bI_2\;\subset\; \bI_1\;\subset\;\Tilt SL_3,$$
where $\bI_j$ contains $T(\lambda)$ iff $\lambda\in (2^j-1)\rho+X^+$. Every thick tensor ideal is prime.
\end{proposition}

\begin{proof}
It follows from \cite[\S 8]{BE} that the semisimplification of $\Tilt \SL_3$ is equivalent to the category of $\mZ$-graded vector spaces, where the natural representation $V$ goes to a generator. By Theorem~\ref{thm:N}, the only non-negligible indecomposable tilting modules are $T(a\omega_1)$ and $T(b\omega_2)$ for $a,b\in\mN$. In other words, the maximal thick tensor ideal comprises the indecomposables for $\rho+X^+$, agreeing with the description of $\bI_1$ in the proposition.
The result then follows from Lemma~\ref{lem:An}, which is applicable since $h=3$, so $2\ge 2h-4$. All thick tensor ideals are prime by \cite[Lemma~2.4.3(1)]{CEOq}.
\end{proof}

\begin{proposition}[Andersen]\label{prop:AnOdd}
The non-zero proper thick tensor ideals in $\Tilt \SL_3$ for $p>2$ form one countable chain
$$\cdots\;\subset\; \bJ_2\;\subset\; \bI_2\;\subset\; \bJ_1\;\subset\; \bI_1\;\subset\;\bN\;\subset\;\Tilt SL_3,$$
where $\bN$ contains $T(\lambda)$ iff $\lambda_1-\lambda_3\ge p-2$, the ideal $\bI_j$ contains $T(\lambda)$ iff $\lambda\in (p^j-1)\rho+X^+$, and $\bJ_j$ contains $T(\lambda)$ iff $\lambda=(p^j-1)\rho+\mu +p^j\nu$ where $\mu\in X_j$ and $\nu\in X^+$ with $\nu_1-\nu_3\ge p-2$. Every thick tensor ideal is prime.
\end{proposition}
\begin{proof}
For $p>3$ this was proved in \cite[Example~15]{An}. The case $p=3$ can be proved similarly, using the upgrade in Lemma~\ref{lem:An}.
\end{proof}

\subsection{Classification}

\begin{theorem}\label{thm:3p} Assume $p>2$.
\begin{enumerate}
\item The only T-indecomposable inductive systems of length two are the $\Phi[2; \ell]$ for $\ell\in \mN\cup\{\infty\}$.
\item The only T-indecomposable inductive systems of length three are $\Phi[3; \infty]$, $\Phi[3]$ and, for $i\in\mZ_{>0}$,
 \[\Phi[3;i]:=\left\{D^{\lambda}\mid  \lambda\in\ppart(\le 3) \quad\mbox{with}\quad \lambda_1-\lambda_2< p^{i}-1\mbox{ or } \lambda_2-\lambda_3<p^{i}-1 \right\},\]
and the union $\Phi[3;i]^\ast$ of $\Phi[3;i]$ with
\[\bigcup_{1<j<p}\left\{D^{\lambda}\mid  \lambda\in\ppart(\le 3) \quad\mbox{with}\quad \lambda_1-\lambda_2<jp^i-1\mbox{ and }\lambda_2-\lambda_3<(p+1-j)p^i-1\right\}.\]
\end{enumerate}
\end{theorem}
\begin{proof}
By Corollary~\ref{cor:inducSL}, the classification reduces to the classification of tensor ideals in $\Tilt \SL_2$ and $\Tilt \SL_3$, see Example~\ref{classSL2} and Proposition~\ref{prop:AnOdd}.
\end{proof}

We point out that in the above theorem, all inductive systems are distinct, except that for $p=3$, we actually have $\Phi[3]=\Phi[3;1]$ due to $V$ then becoming negligible in $\Tilt\SL_3$, see also Corollary~\ref{cor:3}.

\begin{theorem}\label{thm:32} Assume $p=2$.
\begin{enumerate}
\item The only T-indecomposable inductive systems of length two are the $\Phi[2; \ell]$ for $\ell\in \mZ_{>0}\cup\{\infty\}$.
\item The only T-indecomposable inductive systems of length three are $\Phi[3; \infty]$ and, for $i\in\mZ_{>0}$, \[\Phi[3;i]:=\left\{D^{\lambda}\mid  \lambda\in\ppart(\le 3) \quad\mbox{with}\quad \lambda_1-\lambda_2< 2^{i+1}-1\mbox{ or } \lambda_2-\lambda_3<2^{i+1}-1 \right\}.\]
\end{enumerate}
\end{theorem}
\begin{proof}
By Corollary~\ref{cor:inducSL}, the classification reduces to the classification of tensor ideals in $\Tilt \SL_2$ and $\Tilt \SL_3$, see Example~\ref{classSL2} and Proposition~\ref{prop:An}.
\end{proof}

\begin{corollary}\label{cor:3}
If $p=2$, we have
\[\Phi[3]=\Phi[3;1]=\{D^{\lambda}\mid \lambda\in\ppart(\le 3) \quad\mbox{with}\quad \lambda_1-\lambda_2\le 2\mbox{ or } \lambda_2-\lambda_3\le 2 \},\]
if $p=3$, we have
\[\Phi[3]=\Phi[3;1]=\{D^{\lambda}\mid \lambda\in\ppart(\le 3) \quad\mbox{with}\quad \lambda_1-\lambda_2\le 1\mbox{ or } \lambda_2-\lambda_3\le 1 \},\]
while for $p>3$, we have
\[\Phi[3]=\{D^{\lambda}\mid \lambda\in\ppart(\le 3) \quad\mbox{with}\quad \lambda_1-\lambda_3< p-2 \}.\]
\end{corollary}

\subsection{Reduction to the oriented Brauer category}

In this section we prove an analogue of Theorem~\ref{thm:SLGL} but with $\Tilt \SL_m$ replaced by another rigid monoidal category, $\Tilt^\bullet \GL_m$. This version is simpler, which allows us to prove it for tensor ideals, not just thick tensor ideals. Moreover, $\Tilt^\bullet \GL_m$ has a simpler diagrammatic presentation, as a quotient of the (Karoubi envelope of the) oriented Brauer category, which might prove to be useful in future efforts. 

\begin{theorem}\label{thm:GLbw}
\begin{enumerate}
\item Every prime thick tensor ideal in $\Tilt^{\circ}\GL_m$ that does not contain $\bK_m$ must be the preimage of a thick tensor ideal under $\Tilt^\circ\GL_m\hookrightarrow \Tilt^\bullet \GL_m$.
\item Every prime tensor ideal in $\Tilt^{\circ}\GL_m$ that does not contain $\cK_m$ must be the preimage of a tensor ideal under $\Tilt^\circ\GL_m\hookrightarrow \Tilt^\bullet \GL_m$.
\end{enumerate}
\end{theorem}
 We start the proof with a lemma.

\begin{lemma}\label{lem:alphaN}
For any tilting module $T$ (not necessarily indecomposable) in $\Tilt^\bullet \GL_m$, there exists $N\in\mN$ such that $T\otimes T(\alpha+N\omega_m)\in \Tilt^\circ\GL_m$.
\end{lemma}
\begin{proof}
Assume first that $T$ is indecomposable. Then $T$ must be a direct summand of $T(\lambda)\otimes T(\mu)^\ast$, for $\lambda,\mu\in \ppart(\le m)$. We might thus as well write $T=T(\mu)^\ast$. Now set $T'=T(\alpha+N\omega_m)\otimes T(\alpha)\otimes T(\alpha)^\ast$, for some $N\in\mN$. Then
\[T\otimes T'\;=\;T(\alpha)^{\otimes 2}\otimes (\w^m V)^{\otimes N}\otimes \left(T(\alpha)\otimes T(\mu)\right)^\ast.\]
Just as in the proof of Theorem~\ref{thm:SLGL}, we can observe that, for all high enough $N$, the above belongs to $\Tilt^\circ \GL_m$. Since $T(\alpha+N\omega_m)$ is a direct summand of $T'$, the claim follows.

If $T$ is not indecomposable we simply need to consider the maximum of the required $N\in\mN$ for its indecomposable summands.
\end{proof}

\begin{proof}[Proof of Theorem~\ref{thm:GLbw}]
Part (1) can be proved by a simplified version of the proof of Theorem~\ref{thm:SLGL}, or similarly to part (2) below.

For part (2), we consider a prime tensor ideal $\cI$ that does not contain $\cK_m$. We let $\cJ$ be the tensor ideal generated by $\cI$ in $\Tilt^\bullet\GL_m$ and $\cI'$ its intersection with $\Tilt^\circ\GL_m$. Clearly $\cI\subset\cI'$ and we prove this is an equality. An arbitrary morphism $f$ in $\cI'$ can be written as
\[f\;=\;\sum_i g_i\circ (f_i\otimes T_i)\circ h_i\]
where $f_i$ are morphisms in $\cI$ and $T_i\in \Tilt^\bullet \GL_m$. Of course the target (resp. source) of $g_i$ (resp. $h_i$) has to be in $\Tilt^\circ\GL_m$. By Lemma~\ref{lem:alphaN}, there is $N\in\mN$, such that $T_i\otimes T(\alpha+N\omega_m)\in\Tilt^\circ \GL_m$. It follows that $f\otimes T(\alpha+N\omega_m)\in\cI$. Since (the identity morphism of) $T(\alpha+N\omega_m)$ is not in $\cI$, see Proposition~\ref{prop:SLGL}(2) for $\lambda=0$, primeness of $\cI$ shows $f\in\cI$ indeed.
\end{proof}

\subsection{Categorical dimension of ideals}

In \cite[\S 6.2]{Tprime}, the notion of a categorical dimension of an ideal in $\bk S_\infty$ was introduced. Since $\mathrm{char}(\bk)=p>0$, this dimension belongs to $\mF_p\subset\bk$, see \cite[Lemma~6.2.3]{Tprime}.
In \cite[Conjecture~6.3.4]{Tprime} it was conjectured that every T-prime ideal admits a dimension. We prove a potentially stronger result:

\begin{theorem}\label{thm:primedim}
Every weakly T-prime ideal in $\bk S_\infty$ admits a categorical dimension.
\end{theorem}
\begin{proof}
By \cite[Theorem~6.3.2]{Tprime}, it suffices to show that for every weakly T-prime ideal, the associated prime tensor ideal in $\Sym$, see Lemma~\ref{lem:bijections}(1), is the kernel of a tensor functor from $\Sym$ to a rigid $\bk$-SM category. By \ref{sec:SymTilt}, every non-zero prime tensor ideal in $\Sym$ contains the kernel of $\Sym\to \Tilt^\circ \GL_m$ for some $m$, and we choose the minimal such $m$. It thus follows from Theorem~\ref{thm:GLbw}(2) that every non-zero prime tensor ideal in $\Sym$ is the kernel of a tensor functor to a rigid category $\Tilt^\bullet\GL_m/\cJ$, for some $m$ and some tensor ideal $\cJ$.
\end{proof}

\begin{remark}The above yields another approach to, or interpretation of, the classification of maximal ideals in $\bk S_\infty$. Indeed, by Theorem~\ref{thm:primedim} and \cite[2.1.4(2) and 6.3.2(2)]{Tprime}, the only candidates for maximal ideals in $\bk S_\infty$ correspond to the maximal ideals (of negligible morphisms) in $\OB(m)$, for $m\in \mF_p^\times$, and the submaximal ideals (those that are only properly contained in the maximal ideal) of $\OB(0)$. As it turns out, the first $p-1$ ideals are always maximal, whereas the latter adds precisely one maximal ideal when $p=2$.
\end{remark}

\begin{example}
The categorical dimension of the maximal ideal $I[m]$, with $1\le m<p$ if $p>2$ and $m\in\{1,2\}$ if $p=2$, is $m$.
\end{example}

\subsection{Tensor ideals for products of $\SL_2$}

For completeness, we state the analogue of Lemma~\ref{lem:joe2} for characteristic $p>2$, the proof is only very slightly more involved, alternatively see Remark~\ref{rem:proof}. We can define the prime thick tensor ideals $\bI_{n_1,\ldots, n_\ell}$ in $\Tilt \SL_2^{\times \ell}$ just as for $p=2$. The abelian envelopes $\Ver_{p^n}$ of the quotients of $\Tilt\SL_2$ were constructed in \cite{BEO, AbEnv}. Using the results from \cite{BEO} we can, just as in \ref{linkidealVer} realise $\bI_{n_1,\ldots, n_\ell}$ as the kernel of the tensor functor
\[\Tilt SL_2^{\times \ell}\to \Ver_{p^\infty},\quad U_{(a)}\mapsto L_1^{[n_a]},\quad 1\le a \le\ell.\]

We denote by $\xi_n=\exp(\pi i/p^n)$, a primitive $2p^n$-th root of 1.

\begin{lemma}\label{lem:joep}
Assume $p>2$. The following conditions are equivalent on $T\in \Tilt (\SL_2^{\times\ell})$:
\begin{enumerate}
\item $T\in \bI_{n_1,\ldots, n_\ell}$;
\item $\ch T$ is in the ideal of $\Lambda$ generated by the $\left(\frac{x_j^{p^{n_j}}-x_j^{-p^{n_j}}}{x_j-x_j^{-1}}\right)$, for $1\le j\le \ell$;
\item $\ch T$ is in the ideal of $\Lambda$ generated by the $\left(\frac{x_j^{p^{n_j}}-x_j^{-p^{n_j}}}{x_j^{p^{n_j-1}}-x_j^{-p^{n_j-1}}}\right)$, for $1\le j\le \ell$;
\item $\ch T(\xi_{n_1},\xi_{n_2},\ldots,\xi_{n_\ell})=0$.
\end{enumerate}
\end{lemma}

\begin{remark}\label{rem:proof}
There is an alternative proof for Lemmata~\ref{lem:joe2} and~\ref{lem:joep}. Indeed, since we can realise the thick tensor ideal $\bI_{n_1,\ldots,n_{\ell}}$ as the kernel of a tensor functor from $\Tilt \SL_2^{\ell}$ to $\Ver_{p^n}$, for any $n\ge\max\{n_1,\ldots, n_{\ell}\}$. Using positivity of Frobenius-Perron dimension, the result then follows from the verification (see the proof of \cite[Theorem~4.5]{BEO}) that the resulting ring homomorphism
\[K_0^{\oplus}(\Tilt \SL_2^{\times \ell})\to K_0(\Ver_{p^n})\xrightarrow{\FPdim}\mZ[2\cos(\pi/p^n)]\subset\mR,\]
is given by $[T]\mapsto \ch T(\xi_{n_1},\xi_{n_2},\ldots,\xi_{n_\ell}).$
\end{remark}


\subsection{Towards the classification of T-indecomposable inductive systems}

It was conjectured in \cite[Conjecture~5.1.2]{Tprime} that every T-indecomposable inductive system is of the form $\Phi_X$ for $X\in\Ver_{p^\infty}$ (and moreover, that $X\mapsto \Phi_X$ induces a bijection between the Grothendieck monoid of $\Ver_{p^\infty}$ and the set of T-indecomposable inductive systems). We summarise the progress made towards that conjecture resulting from the current paper.

\begin{theorem}\label{thm:conj} The following inductive systems are of the form $\Phi_X$ for $X\in\Ver_{p^\infty}$:
\begin{enumerate}
\item Every T-indecomposable inductive system of length $\le 3$.
\item If $p=2$, the minimal inductive system $\Phi[m]$ of length $m$, for every $m>0$.
\item If $p>2$, the minimal inductive system $\Phi[m]$ of length $m\le p$. 
\item If $p>2$, the minimal inductive system $\Phi[m]$ of length $m\in\{2p-3, 2p-2,3p-3\}$. 
\end{enumerate}
\end{theorem}
\begin{proof} Part (1) for length 1 is trivial and for length 2 it follows immediately from our preliminary results.
Part (2) follows immediately from Theorem~\ref{thm:FullPhim}(1). Part (3) for $m<p$ is \cite[Theorem~4.1.1(1)]{Poly}. Parts (1) for length 3, (3) for $m=p$, and (4) are proved in more detail below. 
\end{proof}

For $p>2$, denote by $L_2^{[n]}\in\Ver_p^n\subset\Ver_{p^\infty}$ the image of $T_2$ under the defining $\Tilt\SL_2\to\Ver_{p^n}$. Hence $L_2^{[n]}=0$ if and only if $p=3$ and $n=1$.

\begin{lemma}\label{lem:Phi3}
\begin{enumerate}
\item Assume $p=2$. Then $\Phi[3;\infty]=\Phi_{\unit^3}$ and $\Phi[3;i]=\Phi_{L_1^{[i+1]}\oplus\unit}$, for $i>0$.
\item Assume $p>2$. Then $\Phi[3;\infty]=\Phi_{\unit^3}$, $\Phi[3]=\Phi_{L_2^{[1]}}$ provided $p>3$,
\[ \Phi[3;i]=\Phi_{\unit\oplus L_1^{[i]}}\quad\mbox{and}\quad \Phi[3;i]^\ast= \Phi_{L_2^{[i+1]}},\quad\mbox{for}\quad i>0.\]

\end{enumerate}
\end{lemma}
\begin{proof}
For $p>3$, it was demonstrated in the proof of \cite[Proposition~7.2.2]{CEOq} that $\bJ_i$ (with $\bJ_0:=\bN$) from Proposition~\ref{prop:AnOdd} is the preimage of $\bI_{i+1}$ under the restriction functor $\Tilt\SL_3\to\Tilt \SL_2$ that sends $V$ to $T_2$. It was also proved that $\bI_i$ in $\Tilt \SL_3$ is the preimage of $\bI_{i}$ in $\Tilt\SL_2$ under the restriction functor $\Tilt\SL_3\to\Tilt \SL_2$ that sends $V$ to $T_1\oplus\unit$. The exact same argument works for $p=3$. The statement in (2) then follows from the construction of the inductive systems via those tensor ideals in $\Tilt\SL_3$.

The proof of (1), follows the same argument.
\end{proof}

For $p>2$ denote the odd line in the category of super-vector spaces $\sVec\subset\Ver_p$ by $\bar{\unit}$.
\begin{lemma}\label{lem:Phip}
For $p>2$, we have $\Phi[p]=\Phi_{\unit\oplus\overline{\unit}}$.
\end{lemma}
\begin{proof}
The inductive system $\Phi_{\unit\oplus\overline{\unit}}$ was described in \cite[\S 3.3]{CEKO}. We can interpret it as the minimal inductive system containing any choice of Specht modules for an infinite chain (for the inclusion order) of  hook Young diagrams. If we restrict to diagrams of size prime to $p$ these Specht modules are simple. For example, $\Phi_{\unit\oplus\overline{\unit}}$ is the minimal inductive system containing the Specht modules labelled by $(j,1^{(p-1)j-1})$ for $j\in\mZ_{>0}$. As observed in \cite[\S 3.3]{CEKO} the latter Specht module is precisely $D^{\beta(p,j-1)}$.
\end{proof}

\begin{lemma}
If $p>2$, we have
\[\Phi[2p-3]=\Phi_{\bar{\unit}\oplus L_1^{[1]}\otimes \bar{\unit}},\quad\Phi[2p-2]=\Phi_{L_1^{[2]}\otimes \bar{\unit}},\quad\Phi[3p-3]=\Phi_{L_2^{[2]}\otimes\bar{\unit}}.\]
\end{lemma}
\begin{proof}
\cite[Lemma~4.16]{BK} shows that the tensor product of $\Phi[m]$, for $m\in\{2p-3,2p-2,3p-3\}$, with the sign module is an inductive system of length 2 or 3. It thus suffices to show that these systems are T-indecomposable, use the above identifications of the T-indecomposable systems of length $\le 3$, and observe that taking the tensor product with the sign module corresponds to $\Phi_X\leadsto\Phi_{X\otimes\bar{\unit}}$. In more detail:

By completing the computation in \cite[Lemma~4.16]{BK}, we find that the tensor product of $D^{\beta(2p-2,ps)}$ with the sign module is given by $D^{(p(p-1)s+p-1, p(p-1)s)}$. Since all indecomposable length 2 inductive systems are T-indecomposable, the classification in Theorem~\ref{thm:3p}(1) shows that the modules $D^{(p(p-1)s+p-1, p(p-1)s)}$ generate $\Phi[2;1]=\Phi_{L_1^{[2]}}$.

Similarly, by completing the computation in \cite[Lemma~4.16]{BK}, the tensor product of $D^{\beta(2p-3,2ps)}$ with the sign module is given by $D^{(2p(p-1)s+p-1, p(p-2)s, p(p-2)s)}$. We explain why the system is T-indecomposable. By substituting $2r+1$ for $s$ and observing that the last node in the first row is always good, we find that the system contains the simple module for partition
\[(2p(p-1)2r+2p-2, p(p-2)2r, p(p-2)2r)+p(p-2)\omega_{3}.\]
It is easy to show that after suitable restriction, the above module contains a constituent labelled by the same partition without the term $p(p-2)\omega_{3}$. Restriction to the appropriate Young subgroup of the latter simple then produces the exterior product of $D^{(2p(p-1)r+p-1, p(p-2)r, p(p-2)r)}$ with itself by \eqref{eq:SWT}, which demonstrates T-indecomposability. Comparison with Theorem~\ref{thm:3p}(2) then shows that the length 3 system we obtained is $\Phi[3;1]$, so that the conclusion follows from Lemma~\ref{lem:Phi3}(2).

Finally, the case $\Phi[3p-3]$ follows from the analogous, but significantly simpler, argument as for the previous case.
\end{proof}

\subsection*{Acknowledgements} The author thanks Alexander Kleshchev and Andrew Mathas for interesting discussions.
The research was partly supported by ARC grants FT220100125 and DP250100762.

\end{document}